\documentclass{amsart}
\usepackage{amsfonts}
\usepackage{amsmath}
\usepackage{amssymb}
\usepackage{graphicx}

\providecommand{\U}[1]{\protect\rule{.1in}{.1in}}
\newtheorem{theorem}{Theorem}
\theoremstyle{plain}
\newtheorem{acknowledgement}{Acknowledgement}

\newtheorem{claim}{Claim}

\newtheorem{corollary}{Corollary}

\newtheorem{definition}{Definition}
\newtheorem{example}{Example}

\newtheorem{lemma}{Lemma}

\newtheorem{proposition}{Proposition}
\newtheorem{remark}{Remark}

\numberwithin{equation}{section}
\input{tcilatex}

\begin{document}
\title[Dissipative systems on lattices]{The energy flow of dissipative
systems on infinite lattices}
\author{Sini\v{s}a Slijep\v{c}evi\'{c}}
\address{Department of Mathematics, Bijenicka 30, Zagreb, Croatia}
\email{slijepce@math.hr}
\urladdr{}
\date{11 January 2013}

\begin{abstract}
We study the energy flow of dissipative dynamics on infinite lattices,
allowing the total energy to be infinite and considering formally gradient
dynamics. We show that in spatial dimensions 1,2, the flow is for almost all
times arbitrarily close to the set of equilibria, and in dimensions $\geq 3$%
, the size of the set with non-equilibrium dynamics for a positive density
of times is two dimensions less than the space dimension. The theory applies
to first and second order dynamics of elastic chains in a periodic or
polynomial potential, chains with interactions beyond the nearest neighbour,
deterministic dynamics of spin glasses, discrete complex Ginzburg-Landau
equation, and others. We in particular apply the theory to show existence of
coarsening dynamics for a class of generalized Frenkel-Kontorova models in
bistable potential.
\end{abstract}

\maketitle

\section{Introduction}

The state space of lattice dynamical systems we study consists of functions $%
u:\mathbb{Z}^{N}\rightarrow \mathbb{R}^{M}$. Here $N$ denotes the dimension
of a lattice, and each lattice point has $M$ degrees of freedom. We develop
here general results on asymptotic behavior of lattice systems which can be
understood as a thermodynamic limit of a family of finite-dimensional,
gradient dynamical systems. The motivating example is the Frenkel-Kontorova
(FK)\ model, with the total energy given formally by%
\begin{equation*}
E(u)=\sum_{i=-\infty }^{\infty }(u_{i}-u_{i-1})^{2}+V(u_{i})\text{,}
\end{equation*}%
where $V$ is a smooth periodic potential and $u=(u_{i})_{i\in \mathbb{Z}}\in 
\mathbb{R}^{\mathbb{Z}}$.\ Its gradient dynamics is given by $%
du_{i}/dt=\partial /\partial u_{i}E(u)$ (see Section 3 for details). The FK\
model has been studied as a reasonable physical model (e.g. for Josephson
junction arrays) with rich features present in more dimensional and\ more
complex systems (see \cite{Baesens98}, \cite{Baesens04}, \cite{Floria96}, 
\cite{Floria05}, \cite{Hu:05}, \cite{Qin:10}, \cite{Slijepcevic99} and
references therein). The FK\ model is an example of three paradigms studied
in this paper: this is an example of an extended, a dissipative, and a
system on a lattice.

The extended systems property in the setting of partial differential
equations means that a particular state of the system is defined on a
unbounded set (typically the entire $\mathbb{R}^{N}$), with no required
convergence to $0$ at infinity. Their physical importance is that such a
large state space is required to contain solutions such as travelling waves,
kinks and multikinks etc. In the setting of lattice systems, we consider
systems of infinitely many ordinary differential equations, and there is no
straightforward reduction to a finite dimensional system (e.g. because of
periodicity of solutions).

We focus here on the dissipative structure of extended systems. For
(unforced)\ systems for which the total energy is finite, the dissipative
structure typically implies existence of a Lyapunov function - typically the
energy functional itself - which is strictly decreasing along any
non-stationary solution, and bounded from below. By the LaSalle principle,
this implies a complete asymptotic description of the dynamics:\ the
dynamics asymptotically converges to the set of equilibria. Also strict
local minima of the Lyapunov function are Lyapunov stable. For extended
dissipative systems, the techniques do not apply as the energy functional is
divergent. As we will see, the dynamics resembles gradient systems, but can
and typically is topologically different.

Th. Gallay and the author have explored the extended dissipative structure
for partial differential equations. In \cite{Gallay:01}, \cite{Gallay:12},
we have shown that a large class of extended dissipative PDEs, including
examples such as the reaction-diffusion equation and the Navier-Stokes
equation, have certain universal features. In short, for spatial dimensions
1,2, the dynamics is for almost all times arbitrarily close to the set of
equilibria. In dimensions $N$ $\geq 3$, the persistent non-equilibrium
dynamics can be at most of the dimension $\mathbb{R}^{N-2}$. The dynamics,
however, topologically typically differs from gradient systems, as in
dimensions $N\geq 1$ $\omega $-limit sets can contain non-equilibrium
points, and in dimensions $N\geq 3$ periodic orbits and more complex
recurrent behavior can exist.

In this paper we explore this in the context of lattice systems. We show
that the same essential conclusions hold, but with modified proofs and the
key energy flux and dissipation estimates. We also show that the conclusions
apply to several examples which are difficult to model as PDEs, such as
lattices with interactions beyond the nearest neighbour, and lattices with
randomly chosen interactions for each pair of lattice points, such as spin
glasses. In \cite{Slijepcevic00}, we already proved some of the qualitative
features for lattices in dimension $N=1$. We here also strengthen the
results from \cite{Slijepcevic00} even in dimension $N=1$, as we provide
quantitative bounds related to all claims, whereas the results in \cite%
{Slijepcevic00} are merely of qualitative nature.

Our approach is axiomatic:\ we specify universal features of the energy flow
of dissipative lattice dynamics in Section 3, and then deduce general
theorems. As in \cite{Gallay:12}, we assume a solution of a system of
equations is a semiflow $\varphi $ on a state space $X$, and then assign to
each $u\in X\,$its energy, dissipation and flux functions $e,d,f$. Whereas
the Lyapunov function associates to each $u\in X$ a single number, in our
approach we associate to each $u\in X$ its energy, flux and dissipation at
each point $\alpha \in \mathbb{Z}^{N}$ of the lattice. We call such systems
the Lattice Extended Dissipative Systems, or shortly Lattice EDS. The key
features are the energy balance inequality, and the property that flux
generates dissipation. We express it by using the notation analogous to
calculus on $\mathbb{R}^{N}$, adapted to the discrete space $\mathbb{Z}^{N}$%
. In Section 4 we then show that many examples satisfy our assumptions,
including elastic chains in any dimension and with various potentials, but
also systems already mentioned, and semi-discretizations of PDEs such as the
discrete complex Ginzburg-Landau equation.

In Section 6 we show that for dimensions $N=1,2$, we can find $R$ large
enough (as a function of $T$), so that the total energy is not increasing on
the cube of radius $R$ (this is not necessarily true if $N\geq 3$). The
proof of this encodes the key ideas and reveals key constraints to the
energy flux for Lattice EDS. In Sections 7 and 8 we develop universal bounds
on the energy flux and dissipation, which we then use in Section 9 to
describe the asymptotics of Lattice EDS. We also use the flux and
dissipation bounds to develop general bounds to relaxation times.

In Section 11 we consider an application of the results to coarsening in a
bistable potential. This has been studied for continuos-space systems by
Eckmann and Rougemont in dimensions $N=1$ and $N=2$ (\cite{Eckmann98}, \cite%
{Rougemont00}). In particular, they have shown that for the real
Ginzburg-Landau equation, topologically a continuous analogue to gradient
FK\ model dynamics without forcing, one can construct orbits such that the
solution $u(x)$ at each point $x$ of the physical space "jumps" infinitely
many times between two stable equilibria. We call this phenomenon \textit{%
coarsening dynamics}. Here we study a class of FK models in dimensions $%
N=1,2 $. We give sufficient conditions on a probability measure $\mu $ on
the state space, so that $\mu $-a.e. initial condition $u:\mathbb{Z}%
^{N}\rightarrow \mathbb{R}$ has coarsening dynamics.\ This means that for
each lattice point $\alpha \in \mathbb{Z}^{N}$, $u(\alpha )$ jumps
infinitely many times between two stable equilibria of the potential as $%
t\rightarrow \infty $. We show this for both the first degree, gradient
dynamics without forcing (analogously to Eckmann and Rougemont), but also
for the second degree dynamics with sufficiently strong damping.

\section{The calculus on lattices}

In this section we recall the notation and elementary results of calculus on
lattices. We use the analogy between the standard results such as partial
integration and the Stokes theorem, with minimal technical difficulties due
to discreteness of the space. We denote by $u,v,w$ the functions $\mathbb{Z}%
^{N}\rightarrow \mathbb{R}$ or $\mathbb{Z}^{N}\rightarrow \mathbb{R}^{M}$.
The letter $N$ will always denote the dimension of the lattice ($d$ being
reserved for dissipation). Greek letters $\alpha ,\beta \,$will denote the
elements of $\mathbb{Z}^{N}$. Let $\varepsilon _{j}\in \mathbb{Z}^{N}$ be
the "basis" for $j=1,...,N$: $(\varepsilon _{j})_{i}=0$ for $i\not=j$, $1$
for $i=j$. The translation $T_{j}$ is defined naturally for $u:\mathbb{Z}%
^{N}\rightarrow \mathbb{R}$ or $u:\mathbb{Z}^{N}\rightarrow \mathbb{R}^{M}$
with $T_{j}u(\alpha )=u(\alpha +\varepsilon _{j})$.

The partial derivatives and differentials for $u:\mathbb{Z}^{N}\rightarrow 
\mathbb{R}$ are now:%
\begin{eqnarray*}
\partial _{j}u(\alpha ) &=&u(\alpha )-u(\alpha -\varepsilon _{j}), \\
\partial _{j}^{\ast }u(\alpha ) &=&u(\alpha +\varepsilon _{j})-u(\alpha
)=T_{j}\partial _{j}u(\alpha ), \\
\nabla u(\alpha )_{j} &=&\partial _{j}u(\alpha )\text{,} \\
\nabla ^{\ast }u(\alpha )_{j} &=&\partial _{j}^{\ast }u(\alpha )\text{,}
\end{eqnarray*}%
where $\nabla u(\alpha ),\nabla ^{\ast }u(\alpha ):\mathbb{Z}^{N}\rightarrow 
\mathbb{R}^{N}$. For $N=1$, we will also use the notation $\delta _{\alpha
},\delta _{\alpha }^{\ast }$. Given $v:\mathbb{Z}^{N}\rightarrow \mathbb{R}%
^{N}$, the divergence is%
\begin{eqnarray*}
\func{div}v(\alpha ) &=&\sum_{j=1}^{N}\partial _{j}v(\alpha )_{j}, \\
\func{div}^{\ast }(\alpha ) &=&\sum_{j=1}^{N}\partial _{j}^{\ast }v(\alpha
)_{j}.
\end{eqnarray*}%
We can also naturally define the Laplacian $\triangle u:\mathbb{Z}%
^{N}\rightarrow \mathbb{R}$ for $u:\mathbb{Z}^{N}\rightarrow \mathbb{R}$ by
any of the following equivalent relations:%
\begin{eqnarray}
\triangle u &=&\func{div}\nabla ^{\ast }u=\func{div}^{\ast }\nabla u
\label{r:cal1} \\
&=&\sum_{j=1}^{N}\partial _{j}\partial _{j}^{\ast }u=\sum_{j=1}^{N}\partial
_{j}^{\ast }\partial _{j}u,  \label{r:cal2}
\end{eqnarray}%
or equivalently $\triangle u(\alpha )=\sum_{j=1}^{N}\left[ u(\alpha
+\varepsilon _{j})+u(\alpha -\varepsilon _{j})-2u(\alpha )\right] $. The
proof of the following partial integration relations is straightforward and
left to the reader.

\begin{lemma}
Let $u,w:\mathbb{Z}^{N}\rightarrow \mathbb{R}$. Then%
\begin{eqnarray}
\partial _{j}^{\ast }(uw) &=&\partial _{j}^{\ast }u\cdot T_{j}v+u\cdot
\partial _{j}^{\ast }v  \label{r:cal3} \\
&=&\partial _{j}^{\ast }u\cdot v+T_{j}u\cdot \partial _{j}^{\ast }v,
\label{r:cal4} \\
\func{div}^{\ast }(u\cdot \nabla w) &=&u\triangle w+\nabla ^{\ast }u\cdot
\nabla ^{\ast }v,  \label{r:cal5} \\
\func{div}(u\cdot \nabla ^{\ast }w) &=&u\triangle w+\nabla u\cdot \nabla v.
\label{r:cal6}
\end{eqnarray}
\end{lemma}

Our key tool will be the analogue of the Stokes theorem. For square
lattices, the natural setting is to consider the cubes $C(r),C^{\ast }(r)$
of the radius $r\in \mathbb{N}$, and their boundaries $\partial
C(r),\partial C^{\ast }(r)$, defined as follows:%
\begin{eqnarray*}
C(r) &=&\{\alpha \in \mathbb{Z}^{N}\text{, }-r\leq \alpha _{j}\leq r-1\text{%
, }j=1,...,N\}\text{,} \\
C^{\ast }(r) &=&\{\alpha \in \mathbb{Z}^{N}\text{, }-r+1\leq \alpha _{j}\leq
r\text{, }j=1,...,N\}\text{,} \\
\partial C(r) &=&\partial C^{\ast }(r)=\{\alpha \in \mathbb{Z}^{N},\text{ }%
|\alpha _{j}|\leq r\text{ for }j=1,...,N\text{, but for some }k\text{, }%
|\alpha _{k}|=r\}\text{.}
\end{eqnarray*}

The normal vectors $n:\partial C(r)\rightarrow \mathbb{R}^{N}$, $n^{\ast
}:\partial C^{\ast }(r)\rightarrow \mathbb{R}^{N}$ are now%
\begin{eqnarray*}
n^{\ast }(\alpha ) &=&\left\{ 
\begin{array}{ll}
\sum_{i=1}^{k}\varepsilon _{j_{i}}, & \alpha _{j_{1}},...,\alpha _{j_{k}}=r,%
\text{ all other }-r+1\leq \alpha _{j}\leq r-1 \\ 
-\varepsilon _{k} & \alpha _{k}=-r\text{, all other }-r+1\leq \alpha
_{j}\leq r \\ 
0 & \text{otherwise,}%
\end{array}%
\right. \\
n(\alpha ) &=&\left\{ 
\begin{array}{ll}
\varepsilon _{k} & \alpha _{k}=r\text{, all other }-r\leq \alpha _{j}\leq r-1
\\ 
-\sum_{i=1}^{k}\varepsilon _{j_{i}}, & \alpha _{j_{1}},...,\alpha
_{j_{k}}=-r,\text{ all other }-r+1\leq \alpha _{j}\leq r-1 \\ 
0 & \text{otherwise.}%
\end{array}%
\right.
\end{eqnarray*}

Note that $|n^{\ast }(\alpha )|,|n(\alpha )|\leq \sqrt{N}$, with the values
higher than $1$ occurring at "corners" (at which also $n^{\ast }(\alpha
),n(\alpha )$ differ).

We push further the analogy with the continuous case, hopefully to ease
reading of the arguments to follow, by often using the integral symbol
instead of the finite sum over a bounded set $A\subset \mathbb{Z}^{N}$:%
\begin{equation*}
\int_{A}u(\alpha )d\alpha :=\sum_{\alpha \in A}u(\alpha ).
\end{equation*}

The discrete space analogue of the Stokes theorem now reads:

\begin{lemma}
\label{l:stokes}Let $v:\mathbb{Z}^{N}\rightarrow \mathbb{R}^{N}$ and $r\geq
1 $. Then%
\begin{eqnarray}
\int_{C^{\ast }(r)}\func{div}v(\alpha )d\alpha &=&\int_{\partial C^{\ast
}(r)}v(\alpha )\cdot n^{\ast }(\alpha )d\alpha ,  \label{r:one} \\
\int_{C(r)}\func{div}^{\ast }v(\alpha )d\alpha &=&\int_{\partial
C(r)}v(\alpha )\cdot n(\alpha )d\alpha .  \label{r:two}
\end{eqnarray}
\end{lemma}

\begin{proof}
First note that by definition 
\begin{equation*}
\sum_{\alpha \in C^{\ast }(r)}\partial _{j}v_{j}(\alpha )=\sum_{\alpha \in
C^{\ast }(r),\alpha _{j}=r}v_{j}(\alpha )-\sum_{\alpha \in C^{\ast
}(r),\alpha _{j}=-r+1}v_{j}(\alpha -\varepsilon _{j})
\end{equation*}%
(all the other terms cancel out). We now sum over $j=1,...,r$, recall the
definition of $n^{\ast }$ and get (\ref{r:one});\ (\ref{r:two}) is analogous.
\end{proof}

\section{The key properties}

In this section we specify key properties of the energy flow of dissipative
systems on lattices. To motivate the definitions to follow, we start with
perhaps the simplest example: the over-damped dynamics of the
one-dimensional FK model, i.e. a system of elastically connected balls in a
periodic potential. A state of the model is given by a function $u:\mathbb{%
Z\rightarrow R}$, and its total energy is a formal sum (typically divergent):%
\begin{equation}
E(u)=\sum_{\alpha \in \mathbb{Z}}\left\{ \frac{1}{2}(u(\alpha )-u(\alpha
-1)-b)^{2}+V(u(\alpha ))\right\} \text{,}  \label{r:lyapunov}
\end{equation}%
where $b\in \mathbb{R}$ is a constant (the characteristic length), and
\thinspace $V:\mathbb{R\rightarrow R}$ is a smooth periodic potential, $%
V(x+1)=V(x)$. The gradient dynamics is formally expressed as $\partial
_{t}u=-\nabla E(u)$, or more precisely with the equations%
\begin{eqnarray}
\partial _{t}u(\alpha ) &=&-2u(\alpha )+u(\alpha -1)+u(\alpha +1)-V^{\prime
}(u(\alpha ))  \label{r:FKone} \\
&=&\triangle u(\alpha )-V^{\prime }(u(\alpha )).  \notag
\end{eqnarray}

It is well known that the equation\ (\ref{r:FKone}) defines a continuos
semiflow $\varphi $ on the space $X_{K}=\{u:\mathbb{Z\rightarrow R}%
,|u(\alpha )-u(\alpha -1)|\leq K$ for all $\alpha \}$, where $K$ is any
positive integer. In particular, local existence and continuity follow from
the standard results on solutions of ordinary differential equations on
Banach spaces, and $X_{K}$ is invariant by the order-preserving property of (%
\ref{r:FKone}) and the existence of stationary solutions with the spacing $K$%
. We consider two topologies on $X_{K}$:\ the first one is induced by the $%
||.||_{\infty }$ norm, and the second one is the induced product topology
(as a subset of $\mathbb{R}^{Z}$). The semiflow $\varphi $ is continuous in
both topologies.

As was noted in the introduction, the dynamics (\ref{r:FKone}) is not
gradient in the strict sense (and the LaSalle principle does not hold), as
the "Lyapunov" function (\ref{r:lyapunov}) is typically not finite.

We can associate to each $u:\mathbb{Z\rightarrow R}$ its local energy $%
e_{u}(\alpha )$ as%
\begin{equation*}
e_{u}(\alpha )=\frac{1}{2}(u(\alpha )-u(\alpha -1))^{2}+V(u(\alpha )),
\end{equation*}%
and then the total energy in the interval $[-r+1,r]=C^{\ast }(r)$, denoted
by $E_{u}(r)$ is equal to (using also the notation from the previous section)%
\begin{equation*}
E_{u}(r)=\int_{C^{\ast }(r)}e_{u}(\alpha )d\alpha =\sum_{\alpha
=-r+1}^{r}\left\{ \frac{1}{2}(u(\alpha )-u(\alpha -1))^{2}+V(u(\alpha
))\right\} \text{.}
\end{equation*}%
We now consider the time evolution of the energy $E_{u}(r)$, and obtain the
energy balance equation:%
\begin{eqnarray*}
\partial _{t}E_{u(t)}(r) &=&\sum_{\alpha =-r+1}^{r}\left\{ -(2u(\alpha
)-u(\alpha -1)-u(\alpha +1))\partial _{t}u(\alpha )+V^{\prime }(u(\alpha
))\partial _{t}u(\alpha )\right\} \\
&&-(u(r+1)-u(r))\partial _{t}u(r)+(u(-r+1)-u(-r))\partial _{t}u(-r) \\
&=&-\sum_{\alpha =-r+1}^{r}(\partial _{t}u(\alpha ))^{2}-\partial _{\alpha
}^{\ast }u(r)\partial _{t}u(r)+\partial _{\alpha }^{\ast }u(-r)\partial
_{t}u(-r)\text{.}
\end{eqnarray*}%
We can now introduce the local energy dissipation $d_{u}$ and flux $f_{u}$
associated to $u:\mathbb{Z\rightarrow R}$ with%
\begin{eqnarray*}
d_{u}(\alpha ) &=&(\partial _{t}u(\alpha ))^{2}, \\
f_{u}(\alpha ) &=&-\partial _{t}u(\alpha )(u(\alpha +1)-u(\alpha ))=\partial
_{t}u(\alpha )\partial _{\alpha }^{\ast }u(\alpha )\text{.}
\end{eqnarray*}%
The energy balance equation can now be written in its integral and local
form, also related by the Stokes theorem (Lemma \ref{l:stokes})%
\begin{eqnarray}
\partial _{t}E_{u(t)}(r) &=&-\int_{C^{\ast }(r)}d_{u(t)}(\alpha )d\alpha
+\int_{\delta C^{\ast }(r)}f_{u(t)}(\alpha )\cdot n^{\ast }(\alpha )d\alpha ,
\label{r:global} \\
\partial _{t}e &=&-d+\func{div}f.  \label{r:local}
\end{eqnarray}

The energy balance equation (\ref{r:local}) is the first key component of
the theory to follow. The second one is the observation that \textit{there
is no energy transport (or energy flux)} \textit{without dissipation}. For
the equation (\ref{r:FKone}) we can without loss of generality assume $V\geq
0$ (otherwise we add a constant).\ With the choice of $e,d,f$ as above, we
have.%
\begin{equation*}
f^{2}\leq 2||e||_{\infty }d\text{.}
\end{equation*}

Motivated by this, we introduce an abstract definition of an extended
dissipative system on a lattice, associated to a continuous semiflow rather
than to a specific set of equations. If $\varphi \,$is a semiflow on a state
space $X$, we will associate to each $u\in X$ the functions describing
energy, dissipation and flux at each point of a lattice $\mathbb{Z}^{N}$.
Thus we have a family of mappings $e,d,f_{1},...,f_{N}:X\rightarrow
l_{\infty }(\mathbb{Z}^{N})$, where $f=(f_{1},...,f_{N})\,$is the flux
vector. Here $l_{\infty }(\mathbb{Z}^{N})$ is as usual the space of bounded
functions $u:$ $\mathbb{Z}^{N}\rightarrow \mathbb{R}$ with the sup-norm $%
||.||_{\infty }$. A point $u\in X$ is stationary for the semiflow $\varphi $%
, if for each $t\geq 0$, $\varphi (u,t)=u$.

We now list the required properties of the functions \thinspace $e,d,f$, as
we have seen satisfied for the equation (\ref{r:FKone}).

\begin{description}
\item[(A1)] $e,d\geq 0$;

\item[(A2)] For each $\alpha \in \mathbb{Z}^{N}$, the functions $X\mapsto
e(\alpha ),d(\alpha ),f_{1}(\alpha ),...,f_{N}(\alpha )$ are continuos as
functions $X\rightarrow \mathbb{R}$; also $t\rightarrow ||e_{u(t)}||_{\infty
},||d_{u(t)}||_{\infty },||f_{u(t)}||_{\infty }$ are continuous as functions 
$\mathbb{R}^{+}\rightarrow \mathbb{R}^{+}$;

\item[(A3)] If $d\equiv 0$ then $u\in X$ is stationary;

\item[(A4)] The energy balance equation holds ($\partial _{t}$ is derivative
in terms of distributions):%
\begin{equation}
\partial _{t}e=-d+\func{div}f;  \label{r:ebe}
\end{equation}

\item[(A5)] There is a non-decreasing function $b:\mathbb{R}^{+}\rightarrow 
\mathbb{R}^{+}$ such that%
\begin{equation}
f^{2}\leq b(||e||_{\infty })\cdot d\text{.}  \label{r:eba}
\end{equation}
\end{description}

\begin{definition}
Assume $X$ is a metrizable space and $\varphi $ a continuous semiflow on $X$%
. We say that $(X,\varphi )\,$\ is a N-dimensional extended dissipative
system on a lattice (\textbf{Lattice EDS}), if there exist functions $%
e,d,f_{1},...,f_{N}:X\rightarrow l_{\infty }(\mathbb{Z}^{N})$, $%
f=(f_{1},...,f_{N})$, satisfying (A1)-(A5).
\end{definition}

In the dimension $N=1$, we will sometimes write (\ref{r:ebe}) as%
\begin{equation*}
\partial _{t}e=-d+\partial _{\alpha }f\text{.}
\end{equation*}

Given a point $u\in X$, we will often denote its semiorbit by $u(t)$, and
the family of functions $e,d,f$ with $e(\alpha ,t),d(\alpha ,t),f(\alpha ,t)$
(evolution of the energy, flux and dissipation along a semiorbit). Note that
then (A2) and continuity of the semiflow imply that $e,d,f_{1},...,f_{N}:%
\mathbb{Z}^{d}\times \lbrack 0,\infty )\rightarrow \mathbb{R}$ are
continuous in the second variable. When not clear to which $u\in X$ we
associate $e,d,f$, we write it as $e_{u},d_{u},f_{u}$, otherwise the
subscript $u$ is omitted.

In the next section, we give a number of examples justifying the definition
of Lattice EDS.

We briefly comment perhaps the most technical assumption (A2). It may seem
that this can be replaced with a simpler requirement that $e,d,f$ are
continuous as functions $X\rightarrow l_{\infty }(\mathbb{Z}^{N})$. This
will however not typically be true in our setting, as we will typically
consider the product topology on $X$ (i.e. topology of pointwise convergence
at lattice points). Thus we require only that $e,d,f$ are continuous locally
(i.e. at one lattice point), and uniformly continuous along a semiorbit,
which will hold in our examples.

In most cases, we will also require that the energy $e$ is uniformly bounded
on $X$, so we introduce the following notion:

\begin{definition}
We say that a Lattice EDS\ is bounded, if $\sup_{u\in X}||e_{u}||_{\infty
}<+\infty $.
\end{definition}

For bounded Lattice EDS, we often denote by $\beta =\sup_{u\in
X}b(||e_{u}||_{\infty })$, and then (\ref{r:eba}) becomes%
\begin{equation}
f^{2}\leq \beta d\text{.}  \label{r:beta}
\end{equation}

This requirement means that the local energy remains uniformly bounded,
while the total energy may still be (and typically is) infinite. This is
typically physically a reasonable requirement for dissipative systems, and
in a number of cases can be demonstrated by either the ordering property of
solutions (i.e. a maximum principle), or by appropriate local energy
estimates. We will discuss it further related to particular examples.

\section{Examples of extended dissipative systems on lattices\label%
{s:examples}}

\subsection{The standard damped Frenkel Kontorova model.}

The set of states of a N-dimensional FK model is the set of functions $f:%
\mathbb{Z}^{N}\rightarrow \mathbb{R}^{N}$. The dynamics is governed by the
elastic force between the particles moving in a smooth periodic potential $V:%
\mathbb{R}^{N}\rightarrow \mathbb{R}$, $V(x+\alpha )=V(x)$ for all $\alpha
\in \mathbb{Z}^{N}$. The damped equations can be thus written (by again
relying on the notation from Section 2):%
\begin{equation}
\lambda \cdot \partial _{tt}u_{j}(\alpha )+\partial _{t}u_{j}(\alpha
)=\triangle u_{j}(\alpha )-\frac{\partial }{\partial x_{j}}V(u(\alpha )),
\label{r:fkgen}
\end{equation}%
where $u:\mathbb{Z}^{N}\rightarrow \mathbb{R}^{N}$, $u(\alpha
)=(u_{1}(\alpha ),u_{2}(\alpha ),...,u_{N}(\alpha ))$, and $\lambda \geq
0\,\ $is a constant. Again without loss of generality we assume $V\geq 0$.
We define the set of functions $f:\mathbb{Z}^{N}\rightarrow \mathbb{R}^{N}$
of bounded width with%
\begin{equation}
Y=\{u:\mathbb{Z}^{N}\rightarrow \mathbb{R}^{N},\text{ }||\nabla u||_{\infty
}<\infty \}\text{,}  \label{r:bw}
\end{equation}%
where $\nabla u(\alpha )$ is a $N\times N$ matrix, $(\nabla u(\alpha
))_{i,j}=\partial _{i}u_{j}(\alpha )$, its norm $|\nabla u(\alpha
)|^{2}=\sum_{i,j=1}^{N}(\partial _{i}u_{j}(\alpha ))^{2}$ and $||\nabla
u||_{\infty }=\sup_{\alpha \in \mathbb{Z}^{N}}|\nabla u(\alpha )|$.

In the gradient case $\lambda =0$, the state space will be $X=Y$, and in
damped the case $\lambda >0$, it will be $X=Y\times l_{\infty }(\mathbb{Z}%
^{N})$, $(u,\partial _{t}u)\in X$. It is well known that the equation (\ref%
{r:fkgen}) generates a continuos semiflow on $X$, The considered topology is
the induced product topology on $Y$ (as a subset of $(\mathbb{R}^{N})^{%
\mathbb{Z}^{N}}$) and in the case $\lambda >0$ the product topology on $X$.
In particular, the global existence of solutions and the invariance of $X$
follows from the fact that $||\nabla u||_{\infty }$ can grow at most
exponentially, thus it can not diverge at some finite time.

We now show that this is a N-dimensional Lattice EDS. We define the energy,
dissipation and flux associated to $u\in X$ with%
\begin{eqnarray*}
e(\alpha ) &=&\frac{\lambda }{2}(\partial _{t}u(\alpha ))^{2}+\frac{1}{2}%
|\nabla u(\alpha )|^{2}+V(u(\alpha )), \\
d(\alpha ) &=&(\partial _{t}u(\alpha ))^{2}, \\
f(\alpha ) &=&\sum_{j=1}^{N}\partial _{t}u_{j}(\alpha )\nabla ^{\ast
}u_{j}(\alpha ),
\end{eqnarray*}%
where $(\partial _{t}u)^{2}=\sum_{j=1}^{N}(\partial _{t}u_{j})^{2}$. We
easily see that in both cases $\lambda =0$, $\lambda >0$ we have $e,d,f\in
l_{\infty }(\mathbb{Z}^{N})$, and that the axioms (A1),\ (A2), (A3)\ hold
trivially. We now have%
\begin{equation*}
\partial _{t}e(\alpha )=\sum_{j=1}^{N}\lambda \partial _{t}u_{j}\cdot
\partial _{tt}u_{j}+\sum_{j=1}^{N}\nabla u_{j}(\alpha )\cdot \partial
_{t}\nabla u_{j}(\alpha )+\sum_{j=1}^{N}\frac{\partial }{\partial x_{j}}%
V(u(\alpha ))\partial _{t}u_{j}(\alpha ).
\end{equation*}%
As by the partial integration (\ref{r:cal6}), the second sum on the
right-hand side is equal to $\sum_{j=1}^{N}\func{div}(\partial
_{t}u_{j}(\alpha )\cdot \nabla ^{\ast }u_{j}(\alpha ))-$ $%
\sum_{j=1}^{N}\partial _{t}u_{j}(\alpha )\triangle u_{j}(\alpha )$, we get
the energy balance equation%
\begin{eqnarray*}
\partial _{t}e(\alpha ) &=&-\sum_{j=1}^{N}(\partial _{t}u_{j}(\alpha
))^{2}+\sum_{j=1}^{N}\func{div}(\partial _{t}u_{j}(\alpha )\cdot \nabla
^{\ast }u_{j}(\alpha )) \\
&=&-d(\alpha )+\func{div}f(\alpha )\text{,}
\end{eqnarray*}%
thus (A4)\ holds (note that $e$ is differentiable for all $t>0$). We can
also easily check (A5):%
\begin{eqnarray*}
f^{2}(\alpha ) &=&\left( \sum_{j=1}^{N}\partial _{t}u_{j}(\alpha )\nabla
^{\ast }u_{j}(\alpha )\right) ^{2}=\sum_{i=1}^{N}\left(
\sum_{j=1}^{N}\partial _{t}u_{j}(\alpha )\partial _{i}^{\ast }u_{j}(\alpha
)\right) ^{2} \\
&\leq &\sum_{j=1}^{N}(\partial _{t}u_{j}(\alpha
))^{2}\sum_{j=1}^{N}(\partial _{i}^{\ast }u_{j}(\alpha ))^{2}\leq d(\alpha
)|\nabla ^{\ast }u(\alpha )|^{2} \\
&\leq &d(\alpha )\sup_{\alpha \in \mathbb{Z}^{N}}|\nabla ^{\ast }u(\alpha
)|^{2}\leq d(\alpha )\frac{1}{N}\sup_{\alpha \in \mathbb{Z}^{N}}|\nabla
u(\alpha )|^{2}\leq \frac{2}{N}d(\alpha )||e||_{\infty }\text{.}
\end{eqnarray*}%
We conclude that (\ref{r:fkgen}) generates a Lattice EDS\ on $X$.

One can consider more specific state spaces. Let us first discuss the case $%
\lambda =0$. For a given constant $K>0$, let $Y_{K}$ be the set of all
states whose semiorbits have uniformly bounded width%
\begin{equation*}
Y_{K}=\{u:\mathbb{Z}^{N}\rightarrow \mathbb{R}^{N},\text{ }||\nabla
u(t)||_{\infty }\leq K\text{ for all }t\geq 0\}\text{.}
\end{equation*}%
Then by definition, $Y_{K}$ is invariant, and (\ref{r:fkgen}) generates a
Lattice EDS on $Y_{K}$ too. We already noted that for $N=1$, all $u\in $ $Y$
are in $Y_{K}$ for a sufficiently large integer $K>0$. Also in any
dimension, all periodic orbits of (\ref{r:fkgen}) are in $Y_{K}$ for $K$
large enough. Also $e$ on $Y_{K}$ is bounded, thus the Lattice EDS is
bounded. If we introduce the equivalence relation $u\sim v$ if $u-v=\alpha
\in \mathbb{Z}^{N}$, then one can show that the dynamics (\ref{r:fkgen}) is
well defined on the quotient space $Y_{K}/\sim $, and induces a Lattice EDS
on $Y_{K}/\sim $ (the definitions of $e,d,f$ are invariant with respect to $%
\sim $). By Tychonoff theorem, $Y_{K}/\sim $ is compact in the induced
product topology. Thus the study of (\ref{r:fkgen}) reduces to a study of a
semiflow on a compact space with a Lattice EDS\ structure.

An analogous construction is possible also in the case $\lambda >0$, by also
uniformly bounding the first derivatives $\partial _{t}u$.

\subsection{Generalized Frenkel-Kontorova model}

We now discuss chains with more general nearest-neighbour interactions. We
consider one-dimensional chains in a larger-dimensional space, i.e. the set
of states comprises of functions $u:\mathbb{Z}\rightarrow \mathbb{R}^{M}$.
We assume that the nearest-neighbour interaction is defined by a smooth
potential $L:\mathbb{R}^{M}\times \mathbb{R}^{M}\rightarrow \mathbb{R}$
satisfying

\begin{description}
\item[(L1)] $L(x,y)=L(x+\alpha ,y+\alpha )$ for all $\alpha \in \mathbb{Z}%
^{M},$

\item[(L2)] $\lim_{|x-y|\rightarrow \infty }L(x,y)=+\infty $.
\end{description}

Clearly the standard 1D Frenkel-Kontorova model is a special case of this.
The gradient equations of motion are%
\begin{equation}
\partial _{t}u(\alpha )=-L_{2}(u(\alpha -1),u(\alpha ))-L_{1}(u(\alpha
),u(\alpha +1)),\,  \label{r:v3}
\end{equation}%
where $L_{1},L_{2}$ denote partial derivatives with respect to the first and
second coordinate. We consider the gradient dynamics only, though the damped
one is\ also an example of a Lattice EDS\ (the choices of $e,d,f$ and
calculation are analogous to the standard case).

We note that this example includes a particularly important family of
potentials $L$: generating functions of $M$-dimensional Tonelli Lagrangian
maps. These are discrete analogues of Tonelli Lagrangian flows on a torus
(see e.g. \cite{Fathi:12}), and can also be understood as time-one maps of
Lagrangian flows. In addition to (L1), (L2), Tonelli Lagrangians satisfy the
following: $L_{12}$ is positive definite, and (L2) is strengthened to be $%
\lim_{|x-y|\rightarrow \infty }L(x,y)/|x-y|=+\infty $. Then the stationary
points of (\ref{r:v3}) correspond by variational principle to orbits of the
associated Lagrangian map, and the study of (\ref{r:v3}) is the study of
formally gradient dynamics of the action functional.

Without loss of generality $L\geq 0$ (by (L1), (L2), $L$ is bounded from
below;\ so if not we add a constant). We define $e,d,f$ associated to a
state $u:\mathbb{Z}\rightarrow \mathbb{R}^{M}$ with%
\begin{eqnarray*}
e(\alpha ) &=&L(u(\alpha -1),u(\alpha )) \\
d(\alpha ) &=&(\partial _{t}u(\alpha ))^{2}, \\
f(\alpha ) &=&-L_{1}(u(\alpha ),u(\alpha +1))\partial _{t}u(\alpha )\text{.}
\end{eqnarray*}

We first check (A4):%
\begin{eqnarray*}
\partial _{t}e(\alpha ) &=&L_{1}(u(\alpha -1),u(\alpha ))\partial
_{t}u(\alpha -1)+L_{2}(u(\alpha -1),u(\alpha ))\partial _{t}u(\alpha ) \\
&=&-(\partial _{t}u(\alpha ))^{2}+L_{1}(u(\alpha -1),u(\alpha ))\partial
_{t}u(\alpha -1)-L_{1}(u(\alpha ),u(\alpha +1))\partial _{t}u(\alpha ) \\
&=&-d(\alpha )+\partial _{\alpha }f(\alpha ).
\end{eqnarray*}

The property (A5) will follow directly from this:

\begin{lemma}
\label{l:defb}Assume $L:\mathbb{R}^{M}\times \mathbb{R}^{M}\rightarrow 
\mathbb{R}$ is continuously differentiable and satisfies (L1), (L2). Then
there exists a non-decreasing function $b:\mathbb{R}^{+}\rightarrow \mathbb{R%
}^{+}$ such that for all $x,y\in \mathbb{R}^{M}$,%
\begin{equation*}
L_{1}(x,y)^{2}\leq b(|L(x,y)|).
\end{equation*}
\end{lemma}

\begin{proof}
We define the following functions $b_{1},b_{2}:\mathbb{R}^{+}\rightarrow 
\mathbb{R}^{+}$ as%
\begin{eqnarray*}
b_{1}(x) &=&\sup_{|b-a|\leq x}L_{1}(a,b)^{2}, \\
b_{2}(y) &=&\sup_{|V(a,b)|\leq y}|b-a|.
\end{eqnarray*}

We see that $b_{1}$ is $<\infty $ by periodicity (L1): the domain of the
supremum can be reduced to a compact set. Also $b_{2}(y)<\infty $ because of
(L2). By definition $b_{1},b_{2}$ are non-decreasing and $L_{1}(a,b)^{2}\leq
b_{1}(|b-a|)$, $|b-a|\leq b_{2}(|V(a,b)|)$. We set $b=b_{2}\circ b_{1}$, and
finally note that the composition of non-decreasing functions is
non-decreasing.
\end{proof}

Again let $X$ be the set of states $u:\mathbb{Z}\rightarrow \mathbb{R}^{M}$
of bounded width as in (\ref{r:bw}), with the induced product topology. By
combining all of the above, we see that (\ref{r:v3}) induces a continuos
semiflow on $X$ which is a Lattice EDS.

\subsection{Lattices with interactions beyond the nearest-neighbour\label%
{s:nearest}}

For simplicity, we discuss one-dimensional lattices with the set of states $%
u:\mathbb{Z}\rightarrow \mathbb{R}$ in the gradient case, though higher
dimensions and the damped case can be analysed analogously as above. Assume $%
C\subset \mathbb{Z}$ is a finite set of lattice distances for which there is
a non-zero interaction potential $L^{\gamma }:\mathbb{R\times R\rightarrow R}
$, $\gamma \in C$, satisfying (L1), (L2) and without loss of generality $%
L^{\gamma }\geq 0$. For simplicity assume $0\not\in C$ (if there is a
periodic site potential, we can add it to any $L^{\gamma }$). The equations
of motion are now%
\begin{equation}
\partial _{t}u(\alpha )=-\sum_{\gamma \in C}(L_{2}^{\gamma }(u(\alpha
-\gamma ),u(\alpha ))+L_{1}^{\gamma }(u(\alpha ),u(\alpha +\gamma ))\text{.}
\label{r:notnearest}
\end{equation}

Let $X$ be again the set of all $u:\mathbb{Z}\rightarrow \mathbb{R}$ of
bounded width, and let $e,d,f$ be%
\begin{eqnarray*}
e(\alpha ) &=&\sum_{\gamma \in C}L^{\gamma }(u(\alpha -\gamma ),u(\alpha )),
\\
d(\alpha ) &=&(\partial _{t}(\alpha ))^{2}, \\
f(\alpha ) &=&-\sum_{\gamma \in C}L_{1}^{\gamma }(u(\alpha ),u(\alpha
+\gamma ))\partial _{t}u(\alpha )\text{.}
\end{eqnarray*}

The proofs of (A1), (A2), (A3)\ are straightforward, and (A4)\ is analogous
to previous examples.

Assume $b^{\gamma }$, $\gamma \in C$ are functions constructed in Lemma \ref%
{l:defb}, associated to $L^{\gamma }$. Then%
\begin{eqnarray*}
f^{2}(\alpha ) &=&(\partial _{t}u(\alpha ))^{2}\left( \sum_{\gamma \in
C}L_{1}^{\gamma }(u(\alpha ),u(\alpha +\gamma ))\right) ^{2} \\
&\leq &|C|(\partial _{t}u(\alpha ))^{2}\sum_{\gamma \in C}L_{1}^{\gamma
}(u(\alpha ),u(\alpha +\gamma ))^{2} \\
&\leq &|C|(\partial _{t}u(\alpha ))^{2}\sum_{\gamma \in C}b^{\gamma
}(||e||_{\infty })\text{.}
\end{eqnarray*}

We see that (A5)\ holds with $b=|C|\sum_{\gamma \in C}b^{\gamma }$. Finally
we note that if for all $\gamma $, $L^{\gamma }$ satisfies the twist
condition $L_{12}^{\gamma }\geq 0$, then the equation (\ref{r:notnearest})
preserves ordering. We can then show that $X_{K}=\{||\partial _{\alpha
}u||\leq K\}$ is invariant for all integer $K$, and (\ref{r:notnearest})
generates a bounded Lattice EDF\ on $X_{K}$.

\subsection{Spin glasses and other random interaction models}

We can also define a lattice in which the interactions are chosen from a
finite family of interactions satisfying (L1), (L2), randomly for each pair
of neighbouring lattice points. We give an example of damped dissipative
dynamics of a system which is closely related to the Edward-Anderson
spin-glass model (\cite{Fisher93}), though numerous other variations are
possible. Assume the site potential is given by%
\begin{equation*}
V(x)=\frac{1}{4}u^{4}-\frac{1}{2}u^{2}+\frac{1}{4}\text{.}
\end{equation*}%
It is bistable, with minima at $x=\pm 1$. Choose parameters $\mu ,\lambda >0$%
. We define two possible nearest-neighbour interactions%
\begin{eqnarray*}
L^{+}(x,y) &=&\frac{\mu }{2}(x-y)^{2}, \\
L^{-}(x,y) &=&\frac{\mu }{2}(x-y-2)^{2}\text{.}
\end{eqnarray*}%
As in spin-glasses models, the interactions either favour the lattice points
being in the same, or the opposite site potential minima. Assume the set of
lattice states is the set of functions $u:\mathbb{Z}^{2}\rightarrow \mathbb{R%
}$. Let $S:\mathbb{Z}^{2}\times \{-\varepsilon _{1},\varepsilon
_{1},-\varepsilon _{2},\varepsilon _{2}\}\rightarrow \{+,-\}$ be a random
function assigning to each neighbouring lattice pair a symmetric interaction
(i.e. satisfying $S(\alpha ,-\varepsilon _{j})=S(\alpha -\varepsilon
_{j},\varepsilon _{j})$). The damped lattice dynamics is then given by%
\begin{equation}
\lambda \cdot \partial _{tt}u(\alpha )+\partial _{t}u(\alpha
)=-\sum_{\varepsilon =\pm \varepsilon _{1},\pm \varepsilon
_{2}}L_{1}^{S(\alpha ,\varepsilon )}(u(\alpha ),u(\alpha +\varepsilon
))+u(\alpha )-u(\alpha )^{3}.  \label{r:spinglass}
\end{equation}

Then (\ref{r:spinglass}) defines a continuous semiflow on $X=\{u:\mathbb{Z}%
^{2}\rightarrow \mathbb{R}$, $||u||_{\infty }\leq 1\}$. Because of the
ordering property of (\ref{r:spinglass}), $X$ is invariant, and it is
compact in the induced product topology.\ Analogously to previous examples,
we can easily show it is a bounded Lattice EDS, with $e,d,f$ chosen as
follows:%
\begin{eqnarray*}
e(\alpha ) &=&\frac{\lambda }{2}(\partial _{t}u)^{2}+\frac{1}{2}%
\sum_{\varepsilon =\pm \varepsilon _{1},\pm \varepsilon _{2}}L^{S(\alpha
,\varepsilon )}(u(\alpha ),u(\alpha +\varepsilon ))+V(u(\alpha )), \\
d(\alpha ) &=&(\partial _{t}u(\alpha ))^{2}, \\
f(\alpha ) &=&\sum_{\varepsilon =\pm \varepsilon _{1},\pm \varepsilon
_{2}}L_{1}^{S(\alpha ,\varepsilon )}(u(\alpha ),u(\alpha +\varepsilon
))\partial _{t}u(\alpha ).
\end{eqnarray*}

\subsection{The discrete complex Ginzburg-Landau equation}

The discrete Ginzburg-Landau equation is the spatial discretization of the
well-known complex Ginzburg-Landau equation (\cite{Aranson02}), also studied
independently, with solutions the time evolutions of functions $u:\mathbb{Z}%
^{N}\rightarrow \mathbb{C}$. The equations are given with%
\begin{equation}
u_{t}(\alpha )=(1+i\lambda )\triangle u+u-(1+i\lambda )|u|^{2}u\text{,}
\label{r:dcgl}
\end{equation}%
where $\lambda >0$ is a real parameter. We consider the dissipative case (in
the general case, the two occurrences of the constant $\lambda $ in (\ref%
{r:dcgl}) are not necessarily the same). Following \cite{Gallay:12}, we
define the auxiliary function $v(\alpha ,t)=u(\alpha ,t)e^{i\lambda t}$, and
then%
\begin{equation}
v_{t}(\alpha )=(1+i\lambda )(\triangle v+v-|v|^{2}v).  \label{r:dgl2}
\end{equation}%
One can check analogously to previous examples that (\ref{r:dgl2}) generates
a Lattice EDS with%
\begin{eqnarray*}
e(\alpha ) &=&\frac{1}{2}|\nabla u(\alpha )|^{2}+\frac{1}{4}(1-|v(\alpha
)|^{2})^{2}, \\
d(\alpha ) &=&\frac{|v_{t}(\alpha )|^{2}}{1+\lambda ^{2}}, \\
f(\alpha ) &=&\func{Re}\left( v_{t}(\alpha )\nabla ^{\ast }\overline{%
v(\alpha )}\right) \text{.}
\end{eqnarray*}%
One can show, by adapting the local energy estimates to the discrete case,
that it is also (by choosing a bounded invariant set $X$) a bounded Lattice
EDS on a compact set.

\section{Energy increasing solutions of Lattice EDS}

The fundamental property of a gradient systems is that its energy functional
(or Lyapunov function)$\ E(u(t))$ strictly decreases along any
non-stationary orbit. A natural question to ask in the framework of extended
dissipative systems on lattices is:\ are there any constraints on the set of
integers $R$, for which%
\begin{equation*}
E(R,T):=\int_{C^{\ast }(R)}e(\alpha ,T)d\alpha
\end{equation*}%
is for a given $T>0$ not smaller than $E(R,0)$. As already noted, we
associate the functions $e,d,f:\mathbb{Z}^{N}\times \lbrack 0,T]$ to a fixed
orbit $u(t)$ of a semiflow in $X$.

The key tool in the following will be the integral form of the energy
balance equality (A4), which follows from (A4) by the Stokes theorem and
integration over $[0,T]$:%
\begin{equation}
E(R,T)-E(R,0)=-\int_{0}^{T}\int_{C^{\ast }(R)}d(\alpha ,t)d\alpha
dt+\int_{0}^{T}\int_{\partial C^{\ast }(R)}f(\alpha ,t)n^{\ast }(\alpha
)d\alpha dt.  \label{r:IntBalance}
\end{equation}

We introduce the notation for the total dissipation in the ball $C^{\ast
}(R) $, and the total flux through $\partial C^{\ast }(R)$ up to the time $%
T>0$:%
\begin{eqnarray}
D(R,T) &=&\int_{0}^{T}\int_{C^{\ast }(R)}d(\alpha ,t)d\alpha dt,
\label{r:totald} \\
F(R,T) &=&\int_{0}^{T}\int_{\partial C^{\ast }(R)}f(\alpha ,t)n^{\ast
}(\alpha )d\alpha dt\text{.}  \label{r:totalf}
\end{eqnarray}%
Now (\ref{r:IntBalance}) can be written as%
\begin{equation}
F(R,T)=E(R,T)-E(R,0)+D(R,T)\text{.}  \label{r:balance2}
\end{equation}%
We will use that to prove the following restriction on the energy increasing
solutions in dimensions $N=1,2$. We denote by $J_{T}$ the following set%
\begin{equation*}
J_{T}=\{R\in \mathbb{N}\text{, }E(R,T)\geq E(R,0)\}\text{.}
\end{equation*}

\begin{theorem}
\label{t:finite}Assume $(X,\varphi )$ is a Lattice EDS, and choose
non-stationary $u\in X$.

(i) If \thinspace $N=1$, then $J_{T}$ is finite.

(ii)\ If $N=2$, then $\sum_{r\in J_{T}}1/r$ is finite.
\end{theorem}

The claim will follow from the following Lemma, whose proof is essentially a
discrete analogue of the separation of variables technique.

\begin{lemma}
\label{l:Grecurrent}Assume $J=\{r_{0}<r_{1}<...\}$ is a set of positive
integers, and assume there is a recurrent sequence $G_{k}$ satisfying the
following for some $\lambda >0$:%
\begin{eqnarray}
G_{0} &=&\varepsilon >0,  \label{r:recurrentone} \\
G_{k} &=&G_{k-1}+\lambda G_{k-1}^{2}/r_{k-1}^{N-1},  \label{r:recurrenttwo}
\\
&&G_{k-1}^{2}/r_{k-1}^{N-1}\text{ is bounded.}  \label{r:recurrentthree}
\end{eqnarray}

If $N=1$, then $J$ is finite. If $N=2$, then $\sum_{r\in J}1/r$ is finite.
\end{lemma}

\begin{proof}
As $G_{k}>0$, dividing (\ref{r:recurrenttwo}) by $G_{k-1}^{2}$ we get 
\begin{equation*}
\frac{\lambda }{r_{k-1}^{N-1}}=\frac{G_{k}}{G_{k-1}^{2}}-\frac{1}{G_{k-1}}%
\text{.}
\end{equation*}

Let $C=\sup_{k\geq 0}G_{k-1}^{2}/r_{k-1}^{N-1}$. Now (\ref{r:recurrenttwo})
implies that $G_{k}/G_{k-1}\leq 1+\lambda C$. From this and (\ref%
{r:recurrenttwo}) we get%
\begin{equation*}
\frac{G_{k}}{G_{k-1}^{2}}-\frac{1}{G_{k-1}}\leq (1+\lambda C)\left( \frac{1}{%
G_{k-1}}-\frac{1}{G_{k}}\right) .
\end{equation*}

Combining all of the above and (\ref{r:recurrentone}) we obtain%
\begin{equation*}
\sum_{r\in \mathcal{R}}\frac{1}{r^{N-1}}\leq \frac{(1+\lambda C)}{\lambda
\varepsilon }\text{.}
\end{equation*}
\end{proof}

We now prove Theorem \ref{t:finite}.

\begin{proof}
Let $\beta =\sup_{0\leq t\leq T}b(||e_{u(t)}||_{\infty })$ (by (A2), this is
finite as $||e_{u(t)}||_{\infty }$ is continuous along an orbit). Then (A5)\
becomes for $u(t),$ $t\in \lbrack 0,T]$, and all $\alpha \in \mathbb{Z}^{N}$%
, 
\begin{equation*}
f^{2}\leq \beta d\text{.}
\end{equation*}%
We now write the integral form of this and get (recall that $|n^{\ast
}(\alpha )|\leq \sqrt{N}$):

\begin{eqnarray*}
F^{2}(R,T) &=&\left( \int_{0}^{T}\int_{\partial C^{\ast }(R)}f(\alpha
,t)n^{\ast }(\alpha )d\alpha dt\right) ^{2} \\
&\leq &T|\partial C^{\ast }(r)|\sqrt{N}\int_{0}^{T}\int_{\partial C^{\ast
}(R)}f^{2}(\alpha ,t)d\alpha dt \\
&\leq &T|\partial C^{\ast }(R)|\sqrt{N}\cdot \beta
\int_{0}^{T}\int_{\partial C^{\ast }(R)}d(\alpha ,t)d\alpha dt\text{.}
\end{eqnarray*}%
We introduce the constant 
\begin{equation}
\omega _{N}=N^{3/2}2^{N},  \label{r:defomega}
\end{equation}%
and then $|\partial C^{\ast }(R)|\sqrt{N}=\omega _{N}R^{N-1}.$As $\partial
C^{\ast }(r)$ are disjoint, we now have%
\begin{equation}
D(R,T)\geq \frac{1}{\omega _{N}\beta T}\sum_{r=1}^{R-1}\frac{F^{2}(r,T)}{%
r^{N-1}}\text{.}  \label{r:dissflux}
\end{equation}

Combining that and (\ref{r:balance2}), we see that for all $R\in J_{T}$,%
\begin{eqnarray}
F(R,T) &\geq &\frac{1}{\omega _{N}\beta T}\sum_{r=1}^{R-1}\frac{F^{2}(r,T)}{%
r^{N-1}}  \notag \\
&\geq &\frac{1}{\omega _{N}\beta T}\sum_{r\in J_{T}\text{, }r\leq R-1\text{.}%
}\frac{F^{2}(r,T)}{r^{N-1}}.  \label{r:relation}
\end{eqnarray}

As $u$ is not stationary, there exists $R_{1}$ large enough so that $%
D(R_{1},T)\geq \varepsilon $ for some $\varepsilon >0$. If $J_{T}$ is
finite, the claim is in both cases proved. If not, by (\ref{r:balance2})
there is $R_{0}\in $ $J_{T}$, $R_{0}\geq R_{1}$ and then $F(R_{0},T)\geq
\varepsilon >0$. Now if $R_{0}=r_{0}<r_{1}<r_{2}<...\,\ $are all elements of 
$J_{T}$ greater or equal than $R_{0}$, we can define a recurrent sequence%
\begin{eqnarray}
G_{0} &=&\varepsilon >0,  \label{r:oneb} \\
G_{k} &=&\frac{1}{\omega _{N}\beta T}\sum_{j=0}^{k-1}\frac{G_{j}^{2}}{%
r_{j}^{N-1}}=G_{k-1}+\frac{1}{\omega _{N}\beta T}\frac{G_{k-1}^{2}}{%
r_{k-1}^{N-1}}\text{.}  \label{r:twob}
\end{eqnarray}

From this and (\ref{r:relation}) we easily deduce by induction that for all $%
k\geq 0$, $G_{k}\leq F(r_{k},T)$. However, by (A2), $t\mapsto
||f_{u(t)}||_{\infty }\,$is continuous, thus bounded on $[0,T]$ by some
constant $C$. By definition then we have for all $r>0$, thus $F(r,T)\leq
C\omega _{N}r^{N-1}$. We deduce that $G_{k}/r_{k}^{N-1}$ is bounded. The
claim now follows from Lemma \ref{l:Grecurrent}.
\end{proof}

If the dimension $N=1$ (as already shown in \cite{Slijepcevic00}), the
theorem \ref{t:finite} can be rephrased as follows:\ given any
non-stationary $u\in X$ and any $T>0$, there exists $R_{0}$ such that for
all $R>R_{0}$,%
\begin{equation}
E(R,T)<E(R,0)\text{.}  \label{r:decreasing}
\end{equation}

Note that the choice of $R_{0}$ is typically non-uniform:\ it depends on $%
T>0 $ and the choice of $u\in X$.

In dimension $N=2$, the set of all $R$ for which (\ref{r:decreasing}) holds
has $0$ density, as the condition that $\sum_{r\in J_{T}}1/r$ is finite is
equivalent to%
\begin{equation*}
|\mathcal{R}(u,T)\cap \{1,..,K\}|=O\left( \frac{K}{\log K}\right) .
\end{equation*}

\begin{remark}
\label{r:finite}In dimensions $N\geq 3$, the rationale above does not bring
any new insight, as $\sum_{r\in \mathbb{N}}1/r^{N-1}$ is always finite.
\end{remark}

\begin{corollary}
\label{c:periodic}If $(X,\varphi )$ is a Lattice EDS in dimensions $N=1$ or~$%
2$, it has no periodic orbits of period $T>0$.
\end{corollary}

\begin{proof}
For periodic orbits of period $T>0$, $E(R,T)=E(R,0)\,\ $must hold for all $R$%
, which is in contradiction with Theorem \ref{t:finite}.
\end{proof}

\section{The flux and dissipation bounds \label{s:fluxbounds}}

In previous section, we investigated constraints on the energy of a cube $%
C^{\ast }(R)$ at a time $T>0$. Here we will use similar rationale to develop
bounds on the total flux $F(R,T)$ through a boundary $\partial C^{\ast }(R)$
of the cube up to a time $T>0$ (see (\ref{r:totalf}) for the definition of $%
F(R,T)$). In this Section we will assume that a Lattice EDS is bounded, with
the constant $\beta $ such that (\ref{r:beta}) holds.\ We fix an orbit $u(t)$%
, and use the notation $e_{0}=||e_{u(0)}||_{\infty }$.

We first state the flux bounds which will be proved in the following. Let $%
T>0$, and $N$ as usual the dimension of Lattice EDS. As we are typically
interested in the asymptotic behavior of our Lattice EDS, we develop the
bounds valid for large $T$ as compared to $R$.

\begin{description}
\item[(i)] If $N=1$, then 
\begin{equation}
\lim \sup_{T\rightarrow \infty }\frac{1}{\sqrt{T}}F(R,T)\leq 2\sqrt{\beta
e_{0}}.  \label{flux1}
\end{equation}

\item[(ii)] If $N=2$, then 
\begin{equation}
\lim \sup_{T\rightarrow \infty }\frac{\log T}{T}F(R,T)\leq 96\sqrt{2}\cdot
\beta .  \label{flux2}
\end{equation}

\item[(iii)] If $N\geq 3$, then%
\begin{equation}
\lim \sup_{T\rightarrow \infty }\frac{1}{T}F(R,T)=N^{5/2}2^{2N-2}\cdot \beta
R^{N-2}.  \label{flux3}
\end{equation}
\end{description}

Let us first discuss the significance of the bounds above. A simple a-priori
analysis can yield only the bound $F(R,T)=O(TR^{N-1})$, i.e. that for
bounded Lattice EDS, the maximal flux through a boundary of a cube $%
C_{R}^{\ast }$ is at most proportional to time and the size of the cube. The
bounds above are much better. In low dimensions $N=1,2$, we see that the
flux can not be proportional to the total time spent. We will see later that
because of that, our Lattice EDS will be for almost all times arbitrarily
close to equilibria. In the dimension $N\geq 3$, we see that the flux can
grow proportionally with time, but only proportionally with $R^{N-2}$. This
will mean that our dynamics may have persistent non-equilibrium behavior
only on the sets of the size $\sim R^{N-2}$, i.e. two dimensions less than
the space dimension.

We will prove the flux bounds by showing that $F(r,T)\leq G_{r}$, where $%
G_{r}$ is a recurrent sequence defined for integer $r$, satisfying

\begin{eqnarray}
G_{r+1} &=&G_{r}-\lambda r^{N-1}+\varepsilon \frac{G_{r}^{2}}{r^{N-1}},
\label{r:recurrent1} \\
&&G_{r}/r^{N-1}\,\ \text{is bounded.}  \label{r:recurrent2}
\end{eqnarray}

As we will see later, here $\lambda \sim e_{0}$, $\varepsilon \sim 1/(\beta
T)$. The analysis of the recurrent sequence for different $N$ is similar to
the proof of Lemma \ref{l:Grecurrent}, but significantly more technical.
Therefore the proof of the following two Lemmas is given in the Appendix.

\begin{lemma}
\label{l:ricatti1} Assume $G_{r}$, $r\geq 1$ satisfies (\ref{r:recurrent1}),
(\ref{r:recurrent2}) for some $\lambda ,\varepsilon >0$.

\begin{description}
\item[(i)] If $N=1$, then%
\begin{equation}
G_{r}\leq \sqrt{\frac{\lambda }{\varepsilon }}.  \label{r:bound1}
\end{equation}

\item[(ii)] If $N\geq 2$, then%
\begin{equation}
G_{r}\leq (N-1)\left( 1+\frac{1}{r}\right) ^{N-2}\frac{r^{N-2}}{\varepsilon }%
+\sqrt{\frac{\lambda }{\varepsilon }}r^{N-1}.  \label{r:boundall}
\end{equation}
\end{description}
\end{lemma}

As we will see in the Appendix, the bound (\ref{r:bound1}) in the case $N=1$
is sharp, and (\ref{r:boundall}) in the cases $N\geq 3$ is sharp up to a
constant. We will be interested in the asymptotic case $\varepsilon
\rightarrow 0$ (i.e. $T\rightarrow \infty )$, so in the case $N=2$ we can
say something more.

\begin{lemma}
\label{l:ricatti2}Assume $G_{r}$, $r\geq 1$ satisfies (\ref{r:recurrent1}), (%
\ref{r:recurrent2}). Assume also $N=2$ and $r\sqrt{\lambda \varepsilon }\leq
1/2$. Then%
\begin{equation}
G_{r}\leq \frac{12}{-\varepsilon \log (2r^{2}\lambda \varepsilon )}\text{.}
\label{r:bound2}
\end{equation}
\end{lemma}

We will use it to prove the following general flux bounds:

\begin{description}
\item[(i)] If $N=1$, then 
\begin{equation}
F(R,T)\leq 2\sqrt{\beta e_{0}T}.  \label{r:general1}
\end{equation}

\item[(ii)] If $N=2$ and $64R^{2}e_{0}\leq \omega _{2}\beta T$, then 
\begin{equation}
F(R,T)\leq \frac{12\omega _{2}\beta T}{\log (\omega _{2}\beta
T/(64e_{0}R^{2}))}.  \label{r:general2}
\end{equation}

\item[(iii)] If $N\geq 2,$ then%
\begin{equation}
F(R,T)\leq (N-1)\left( 1+\frac{1}{R}\right) ^{N-2}\omega _{N}R^{N-2}\beta
T+2^{N}\sqrt{\omega _{N}\cdot N}\cdot R^{N-1}\sqrt{e_{0}\beta T}\text{.}
\label{r:general3}
\end{equation}
\end{description}

\begin{theorem}
Assume a bounded Lattice EDS is given, with the constants $\beta $ and $%
e_{0} $ as above. Then for all $T>0$, $R\geq 1$, (\ref{r:general1}), (\ref%
{r:general2}) and (\ref{r:general3}) hold.
\end{theorem}

\begin{proof}
The first step is writing again the energy balance equation. As $|C^{\ast
}(R)|=2^{N}R^{N}$ and $e_{0}=||e_{u(0)}||_{\infty }$, we have $E(R,0)\leq
2^{N}R^{N}e_{0}$. Inserting that and $e\geq 0$ in (\ref{r:balance2}) we get%
\begin{equation}
F(R,T)\geq -2^{N}R^{N}e_{0}+D(R,T)\text{.}  \label{r:balance3}
\end{equation}

As in the proof of Theorem \ref{t:finite}, (\ref{r:dissflux}) holds.
Inserting it into (\ref{r:balance3}), we get%
\begin{eqnarray*}
F(R,T) &\geq &-2^{N}R^{N}e_{0}+\frac{1}{\omega _{N}\beta T}\sum_{r=1}^{R-1}%
\frac{F^{2}(r,T)}{r^{N-1}} \\
&\geq &-2^{2N}e_{0}N\sum_{r=1}^{R-1}r^{N-1}+\frac{1}{\omega _{N}\beta T}%
\sum_{r=1}^{R-1}\frac{F^{2}(r,T)}{r^{N-1}}\text{,}
\end{eqnarray*}%
where we used 
\begin{equation*}
R^{N}\leq 2^{N}(R-1)^{N}\leq N2^{N}\int_{0}^{R-1}r^{N-1}dr\leq
N2^{N}\sum_{r=1}^{R-1}r^{N-1}.
\end{equation*}%
Now we can set 
\begin{eqnarray*}
\lambda  &=&2^{2N}e_{0}N, \\
\varepsilon  &=&1/(\omega _{N}\beta T).
\end{eqnarray*}%
We define the recurrent sequence for some given $G_{1}\in \mathbb{R}$:%
\begin{equation}
G_{R}=-\lambda \sum_{r=1}^{R-1}r^{N-1}+\varepsilon \sum_{r=1}^{R-1}\frac{%
G_{r}^{2}}{r^{N-1}}.  \label{r:gdef}
\end{equation}

From this, it is easy to see that for all $r\geq 1$ we have%
\begin{equation*}
G_{r+1}=G_{r}-\lambda r^{N-1}+\varepsilon \frac{G_{k-1}^{2}}{r^{N-1}}.
\end{equation*}

Now, if for some $r_{0}$, $G_{r_{0}}=F(r_{0},T),$ inductively we have that
for all $r\geq r_{0}$, $G_{r}\leq F(r,T)$. We can also see that $%
F(r,T)/r^{N-1}$ is bounded, as the flux $||f||_{\infty }$ is uniformly
bounded on $[0,T]$ along the semiorbit $u(t)$ by (A2). Thus we have 
\begin{equation*}
F(r,T)\leq \omega _{N}r^{N-1}\sup_{t\in \lbrack 0,T]}||f_{u(t)}||_{\infty }.
\end{equation*}

We conclude that $\sup_{r\geq 0}G_{r}/r^{N-1}<\infty $. We finally apply
Lemmas \ref{l:ricatti1} and \ref{l:ricatti2}, and insert the definitions of $%
\lambda ,\varepsilon $ into (\ref{r:bound1}), (\ref{r:boundall}) and (\ref%
{r:bound2})\ respectively.
\end{proof}

The asymptotic behavior for large $T$ as compared to $R$ now straightforward:

\begin{corollary}
Assume a bounded Lattice EDS is given, with the constants $\beta $ and $%
e_{0} $ as above. Then (\ref{flux1}), (\ref{flux2}) and (\ref{flux3}) hold.
\end{corollary}

\begin{proof}
We insert into (\ref{r:general1}), (\ref{r:general2}) and (\ref{r:general3})
the definition (\ref{r:defomega}) of $\omega _{N}$, and in the case $N\geq 3$
also $\left( 1+1/R\right) ^{N-2}\leq 2^{N-2}$.
\end{proof}

We can now deduce the following dissipation bounds.

\begin{corollary}
\label{t:dissipation}Assume a bounded Lattice EDS is given, with the
constants $\beta $ and $e_{0}$ as above. Then for any integer $R\geq 1$

\begin{description}
\item[(i)] If $N=1$, then%
\begin{equation}
\lim \sup_{T\rightarrow \infty }\frac{1}{\sqrt{T}}\int_{0}^{T}\int_{C^{\ast
}(R)}d(\alpha ,t)d\alpha dt=O(\sqrt{\beta e_{0}}).  \label{r:done}
\end{equation}

\item[(ii)] If $N=2$, then%
\begin{equation}
\lim \sup_{T\rightarrow \infty }\frac{\log T}{T}\int_{0}^{T}\int_{C^{\ast
}(R)}d(\alpha ,t)d\alpha dt=O(\beta ).  \label{r:dtwo}
\end{equation}

\item[(iii)] If $N\geq 3$, then%
\begin{equation}
\lim \sup_{T\rightarrow \infty }\frac{1}{T}\int_{0}^{T}\int_{C^{\ast
}(R)}d(\alpha ,t)d\alpha dt=O(\beta R^{N-2}).  \label{r:dthree}
\end{equation}
\end{description}
\end{corollary}

\begin{remark}
The notation $F(x)\leq O(G(x))$ means that there exists a constant $C>0$
such that for all $x$, $F(x)\leq CG(x)$ (and not, as is more common, $%
|F(x)|\leq CG(x)$).
\end{remark}

\begin{proof}
Let $N=1$. The claim can be written as%
\begin{equation*}
\lim \sup_{T\rightarrow \infty }\frac{1}{\sqrt{T}}D(R,T)=O(\sqrt{\beta e_{0}}%
).
\end{equation*}%
By the energy balance (\ref{r:balance3}),%
\begin{equation*}
D(R,T)\leq F(R,T)+2^{N}R^{N}e_{0}.
\end{equation*}

Applying now (\ref{flux1}), we get%
\begin{equation*}
\lim \sup_{T\rightarrow \infty }\frac{1}{\sqrt{T}}D(R,T)\leq \lim
\sup_{T\rightarrow \infty }\frac{1}{\sqrt{T}}F(R,T)=O(\sqrt{\beta e_{0}})%
\text{.}
\end{equation*}

The cases $N=2$, $N\geq 3$ are analogous: we apply (\ref{flux2}), (\ref%
{flux3}) instead.
\end{proof}

These bounds on dissipation have important interpretations. In dimensions $%
1,2$, we see that for an average $t>0$, the total dissipation on any chosen
cube $C^{\ast }(R)$ at the time $t$ is arbitrarily close to $0$. We can
write it shortly%
\begin{equation}
\lim \sup_{T\rightarrow \infty }\frac{1}{T}D(R,T)=0\text{.}
\label{r:dissbound12}
\end{equation}%
We will see in the next section that it means that any semiorbit of a
bounded Lattice EDS stays for almost all times arbitrarily close to the set
of equilibria.

In dimensions $N\geq 3$, this conclusion does not necessarily hold. We see,
however, that the dynamics which can persist on the set of times of positive
density is concentrated on a set of two dimensions smaller than the
dimension of the physical lattice space. For example, in dimension $N=3$,
the non-stationary dynamics of dissipative lattice dynamics which can
persist is at most one-dimensional, thus essentially concentrated on a
finite set of lattice points.

\section{Convergence to equilibria in dimensions 1,2}

Now we deduce dynamical consequences of the bounds on flux and dissipation
deduced in the previous section. We focus for now to dimensions $N=1,2$. We
will assume that our Lattice EDS\ is bounded.

Also we assume that the Lattice EDS semiflow $(X,\varphi )$ is given on a
compact space $X$. This is not overly restricting: we have shown that, for
example in the case of the FK model, all orbits of bounded width are
relatively compact in the quotient space $X=Y/\sim $, where with $\sim $ we
identify all configurations which are integer translates. The topology here
is the induced product topology, i.e. the topology of pointwise convergence
on lattice points.

We can thus study the continuos semiflows $\varphi $ on compact metrizable
spaces $X$, with the structure of a bounded Lattice EDS. In this setting, we
are interested on the structure of $\omega $-limit sets of any $u\in X$,
i.e. the set of limit points of $u(t)$ as $t\rightarrow \infty $ in the
chosen topology on $X$.

We first show that the average time a flow spends outside of any
neighbourhood of stationary points is $0$. Denote by $\mathcal{E\subset }X$
the set of all equilibria (or stationary points) of the semiflow $\varphi $,
that means all points $u\in X$ such that for all $t\geq 0$, $u(t)=\varphi
(u,t)=u.$ If $\mathcal{U\subset X}$ is any set, then $1_{\mathcal{U}%
}:X\rightarrow \mathbb{R}$ is as usual its characteristic function, $1_{%
\mathcal{U}}(v)=1$ if $v\in \mathcal{U}$, otherwise $1_{\mathcal{U}}(v)=0$.

\begin{proposition}
\label{p:neighborhood}Assume $(X,\varphi )$ is a bounded Lattice EDS, assume 
$X$ is compact, and let $N=1$ or $2$. Let $\mathcal{U}$ be any open
neighbourhood of the set of stationary points $\mathcal{E}$. Then%
\begin{equation}
\lim_{T\rightarrow \infty }\frac{1}{T}\int_{0}^{T}1_{\mathcal{U}}(u(t))dt=1%
\text{.}  \label{r:timespent}
\end{equation}

In particular, $\mathcal{S}$ is non-empty.
\end{proposition}

\begin{proof}
As $X$ is compact, the complement $\mathcal{U}^{c}\,$is also compact. Choose 
$u\in $ $\mathcal{U}^{c}$. As $u$ is not stationary, by (A3) there exists $%
R(u)$ large enough and $\varepsilon (u)>0$ such that%
\begin{equation*}
\int_{C^{\ast }(R(u))}d_{u}(\alpha )d\alpha \geq 2\varepsilon (u)>0\text{.}
\end{equation*}

By (A2), that means continuity of $u\mapsto d_{u}(\alpha )$, there exists an
open neighborhood $\mathcal{V}(u)$ of $u\in $ $\mathcal{U}^{c}$ such that
for each $v\in \mathcal{V}(u)$,%
\begin{equation*}
\int_{C^{\ast }(R(u))}d_{v}(\alpha )d\alpha \geq \varepsilon (u).
\end{equation*}

The family $\mathcal{V}(u)$, $u\in $ $\mathcal{U}^{c}$ is an open cover of
the compact set $\mathcal{U}^{c}$, thus it has a finite subcover $\mathcal{V}%
(u_{1}),...,\mathcal{V}(u_{k})$. Let $\varepsilon =\min \{\varepsilon
(u_{1}),...,\varepsilon (u_{k})\}$, $R=\max \{R(u_{1}),...,R(u_{k})\}$. Then
for any $u\in \mathcal{U}^{c}$,%
\begin{equation*}
\int_{C^{\ast }(R)}d_{u}(\alpha )d\alpha \geq \varepsilon >0\text{.}
\end{equation*}

We now have%
\begin{eqnarray*}
\lim \sup_{T\rightarrow \infty }\frac{1}{T}\int_{0}^{T}1_{\mathcal{U}%
^{C}}(u(t))dt &\leq &\lim \sup_{T\rightarrow \infty }\frac{1}{T}\int_{0}^{T}%
\frac{1}{\varepsilon }1_{\mathcal{U}^{C}}(u(t))\int_{C^{\ast
}(R)}d_{u(t)}(\alpha )d\alpha dt \\
&\leq &\frac{1}{\varepsilon }\lim \sup_{T\rightarrow \infty }\frac{1}{T}%
D(R,T),
\end{eqnarray*}%
which is by (\ref{r:dissbound12}) equal to $0$.
\end{proof}

The LaSalle principle for gradient systems states that the $\omega $-limit
set of each point consists of stationary points. Our weaker version of this
for Lattice EDS is as follows.

\begin{proposition}
\label{p:contains}If $(X,\varphi )$ is a bounded Lattice EDS, assume $u(t)$
for some $u\in X$ is relatively compact, and let $N=1$ or $2$. Then $\omega
(u)$ contains a stationary point.
\end{proposition}

\begin{proof}
We assume without loss of generality that $X$ is equal to the compact
closure of $u(t)$, thus compact. By Proposition \ref{p:neighborhood}, $u(t)$
visits any neighborhood of $\mathcal{E}$ infinitely many times. By
compactness, $u(t)$ than has a limit point in $\mathcal{E}$.
\end{proof}

We recall the notion of uniform recurrence, as the strongest notion of
recurrence which occurs in general dynamical systems. A point $x\in X$ is
uniformly recurrent with respect to a continuous semiflow $\varphi $ on a
metrizable space $X$, if for each open neighbourhood $U$ of $x$, the set of
return times $\tau =\{t\geq 0$\ $|\;\varphi (t,x)\in U\}$ satisfies the
following:

\begin{description}
\item[R] There exists $n\geq 1$ and $t_{1},...,t_{n}$ so that $[0,+\infty
)=\cup _{k=1}^{n}\{t\geq 0,t+t_{k}\in U\}$.
\end{description}

It is easy to show that Proposition \ref{p:neighborhood} implies that the
set of uniformly recurrent points of a bounded Lattice EDS coincides with
the set of equilibria $\mathcal{E}$ (see e.g. Lemma 5.6 in \cite{Gallay:01}%
). We also recall the relationship between minimal sets and uniformly
recurrent points (see Furstenberg \cite{Furstenberg:81}, Lemma 1.14 and
Theorem 1.15):

\begin{theorem}
\label{t:furstenberg}Given a continuous semiflow on a compact metrizable
space, every minimal closed positively invariant set consists of uniformly
recurrent points.
\end{theorem}

We can now summarize the key dynamical systems properties of bounded Lattice
EDS\ on compact sets in dimensions 1,2.

\begin{theorem}
\label{t:main}Assume $(X,\varphi )$ is a $1$ or $2$-dimensional bounded
Lattice EDS, and assume $X$ is compact. Then

\begin{description}
\item[(i)] $(X,\varphi )$ has no periodic orbits of period $T>0$;

\item[(ii)] The $\omega $-limit set of each $u\in X$ contains an equilibrium;

\item[(iii)] The only uniformly recurrent points are equilibria;

\item[(iv)] The only minimal, closed positively invariant sets consist of
single equilibria;

\item[(v)] All invariant (Borel probability)\ measures are supported on the
set of equilibria.
\end{description}
\end{theorem}

\begin{proof}
We proved (i) in Corollary \ref{c:periodic}, and (ii)\ in Proposition \ref%
{p:contains}. The definition of uniform recurrence and (\ref{r:timespent})
imply (iii);\ (iv)\ follows from (iii) by Theorem \ref{t:furstenberg}.
Relation (\ref{r:timespent}) and the Birkhoff ergodic theorem imply (v).
\end{proof}

\textbf{Discussion of dimensions }$\mathbf{N\geq 3}$. A natural question to
ask is whether some claims of Theorem \ref{t:furstenberg} can be extended to
dimensions $N\geq 3$. We have already noted in Remark \ref{r:finite} that
the constraints to the flux transfer in dimensions $N\geq 3$ are
qualitatively different, as the sum $\sum_{r\in \mathbb{N}}1/r^{N-1}$ is
finite.

Indeed, in dimensions $N\geq 3$ one can construct energy increasing orbits
and periodic orbits, thus contradicting all the claims of Theorem \ref%
{t:main}. The examples in the continuous space case have been constructed in 
\cite{Gallay:01}, and can be adapted to the discrete space case.

We believe, however, that the asymptotic bound on dissipation $D(R,T)\sim
TR^{N-2}$ as $T\gg R$ in dimensions $N\geq 3$ can imply interesting
conclusions for particular equations. We suggest in Section \ref{s:app1} an
application to coarsening dynamics in bistable potential.

\textbf{Structure of the }$\omega $\textbf{-limit set in dimensions }$1,2$%
\textbf{. }Another important question is whether one can describe in some
detail the structure of $\omega $-limits set of compact, bounded Lattice
EDS\ in low dimensions. One can show that even in dimension $1$, for
compact, bounded Lattice EDS, the $\omega $-limit set can contain
non-stationary points, so Lattice EDS indeed differ from the gradient
systems and lattices on finite domains. As in the continuous case (\cite%
{Gallay:12}), our example is inspired by the coarsening dynamics of the real
Ginzburg-Landau equation studied by Eckmann and Rougemont, and is outlined
in section \ref{s:app1}.

\section{Bounds on relaxation times \label{s:relaxationtimes}}

We can use the flux bounds (\ref{r:general1}), (\ref{r:general2}) and (\ref%
{r:general3}) to calculate upper bounds for an arbitrary orbit to relax $%
\varepsilon $-close to the set of equilibria $\mathcal{E}$ on a cube of size 
$r$. Given a bounded Lattice EDS, let $\mathcal{U}_{\varepsilon ,r}\,\ $be
the set of all $u\in X$ such that for all $\alpha \in C^{\ast }(r)$, $%
d(\alpha )<\varepsilon $. Then by (A3),%
\begin{equation*}
\mathcal{E}=\bigcap_{\varepsilon >0\text{, }r\geq 1}\mathcal{U}_{\varepsilon
,r}\text{.}
\end{equation*}

Let $t_{\varepsilon ,r\,}$be the first time such that $u(t)\in \mathcal{U}%
_{\varepsilon ,r}$. Let $\beta $, $e_{0}=||e_{u(0)}||_{\infty }$ be as in
the Section \ref{s:app1}. We will show that:

\begin{description}
\item[(i)] For $N=1$, $t_{\varepsilon ,r\,}\ll e_{0}\beta /(\varepsilon
^{2}r^{2})$,

\item[(ii)] For $N=2$, $t_{\varepsilon ,r\,}\ll \exp (c_{2}\beta
/(\varepsilon r^{2})),$

\item[(iii)] For $N\geq 3$, 
\begin{equation*}
t_{\varepsilon ,r\,}\ll \frac{e_{0}\beta }{r^{2}(c_{N}\beta
/r^{2}-\varepsilon )^{2}},
\end{equation*}%
whenever the denominator is $>0$.
\end{description}

Here $c_{N}$ are constants depending only on the dimension $N$, and are
evaluated below.

If $u(t)\not\in \mathcal{U}_{\varepsilon ,r}$ on $[0,T]$, the energy balance
equality, $e\geq 0$ and $|C^{\ast }(r)|=2^{N}r^{N}$ imply, after integrating
over $[0,T]$:%
\begin{equation}
F(r,T)+e_{0}2^{N}r^{N}\geq D(r,T)\geq \varepsilon 2^{N}r^{N}T\text{.}
\label{r:inequality}
\end{equation}

\textbf{The case }$N=1$\textbf{.} It is easy to check that, if $AT-B\sqrt{T}%
-C\geq 0$ for some constants $A,B,C>0$, then 
\begin{equation}
\sqrt{T}<B/A+\sqrt{C/A}.  \label{r:quadratic}
\end{equation}%
We combine that with (\ref{r:general1}) and (\ref{r:inequality}) and get%
\begin{equation*}
t_{\varepsilon ,r}<\frac{e_{0}\beta }{\varepsilon ^{2}r^{2}}+O\left( \frac{%
\sqrt{\beta }}{\varepsilon ^{3/2}r}\right) \text{.}
\end{equation*}

\textbf{The case }$N=2$\textbf{. }Again we first consider a general
inequality $AT/(\log T-B)+C\geq DT$, for some constants $A,B,C,D>0$, which
implies%
\begin{equation*}
T\log T<\left( \frac{A}{D}+B\right) T+\frac{C}{D}\log T\text{,}
\end{equation*}%
thus%
\begin{equation*}
T<\max \left\{ \exp \left( \frac{A}{2D}+\frac{B}{2}\right) ,\frac{C}{2D}%
\right\} .
\end{equation*}%
From (\ref{r:general2}) and (\ref{r:inequality}) in the case $B=\log
(64e_{0}r^{2}/(\omega _{2}\beta ))>0\Longleftrightarrow 64e_{0}r^{2}>\omega
_{2}\beta $, we have%
\begin{equation*}
t_{\varepsilon ,r}<\max \left\{ \frac{32e_{0}r^{2}}{\omega _{2}\beta }\exp
\left( \frac{2\omega _{2}\beta }{\varepsilon r^{2}}\right) ,\frac{e_{0}}{%
2\varepsilon }\right\} .
\end{equation*}%
(the case $B<0$ results with only a small order correction, and can be
evaluated easily). We see that in dimension 2, we have two relaxation
timescales:\ the fast relaxation for $\varepsilon $ of similar magnitude as
the locally available energy $e_{0}$, and potentially exponentially long
relaxation time for $\varepsilon \ll e_{0}$.

\textbf{The case }$N\geq 3\mathbf{.}$ Similarly as above, we use (\ref%
{r:general3}), (\ref{r:inequality}) and (\ref{r:quadratic}) and get%
\begin{equation*}
t_{\varepsilon ,r}<\frac{N\omega _{N}e_{0}\beta }{r^{2}\left( N\omega
_{N}\beta /(4r^{2})-\varepsilon \right) ^{2}}+O\left( \frac{e_{0}\sqrt{\beta 
}}{r\left( N\omega _{N}\beta /(4r^{2})-\varepsilon \right) ^{3/2}}\right)
\end{equation*}%
if $\varepsilon >N\omega _{N}\beta /(4r^{2})$. If $\varepsilon \leq N\omega
_{N}\beta /(4r^{2})$, the orbit does not necessarily ever enter $\mathcal{U}%
_{\varepsilon ,r}$.

\section{Application example:\ Coarsening in bistable potential \label%
{s:app1}}

\subsection{Introduction and setting.}

Eckmann and Rougemont studied in \cite{Eckmann98} the real Ginzburg-Landau
equation%
\begin{equation}
u_{t}=u_{xx}+u-u^{3}  \label{r:rgle}
\end{equation}%
(a continuous analogue of an elastic 1d lattice in a polynomial potential),
also as an extended system, i.e. considering time evolution of
configurations $u:\mathbb{R}$ $\rightarrow \lbrack -1,1]$ not necessarily
vanishing at infinity. They were interested in dynamics of multikink
solutions: evolution of initial configurations $u(0)$ which "jump" between
minima $\pm 1$ infinitely many times. They then have shown that for selected
initial conditions, for each $x\in \mathbb{R}$, $u(x)$ jumps infinitely many
times between $\pm 1$. They also argued that the resulting dynamics of
lengths of intervals for which a configuration $u$ lies in the same minimum
approximates well the abstract Bray and Derrida coarsening dynamics model 
\cite{Bray95}. Rougemont in \cite{Rougemont00} studied behavior and
disappearance of selected initial "droplet" conditions for the analogue of (%
\ref{r:rgle}) in dimension $2$.

Here we develop a lattice equivalent of this model and behavior, in somewhat
more general setting. In particular, we will also discuss second-order
damped dynamics in dimensions $1,2$.

Assume $L:\mathbb{R}\times \mathbb{R}\rightarrow \mathbb{R}$ is a smooth
function (at least $C^{2}$)\ satisfying (L1), (L2) from Section \ref%
{s:examples}, and in addition the following (satisfied e.g. by the standard
1d FK\ model):

\begin{description}
\item[(L3)] The pair $(0,0)$ is the strict global minimum of $L$;

\item[(L4)] The partial derivative $L_{12}\geq 0$;

\item[(L5)] $L(x,y)=L(-x,-y).$
\end{description}

We will consider the dynamics on the $N$-dimensional, scalar valued lattice,
with the formal energy associated to $u:\mathbb{Z}^{N}\rightarrow \mathbb{R}$%
:%
\begin{equation*}
E(u)=\sum_{\alpha \in \mathbb{Z}^{N}}\sum_{j=1}^{N}L(u(\alpha ),u(\alpha
+\varepsilon _{j}))\text{.}
\end{equation*}

We thus discuss the dynamics%
\begin{equation}
\lambda \cdot \partial _{tt}u(\alpha )+\partial _{t}u(\alpha
)=-\sum_{j=1}^{N}(L_{2}(u(\alpha -\varepsilon _{j}),u(\alpha
)+L_{1}(u(\alpha ),u(\alpha +\varepsilon _{j})),  \label{r:overdamped}
\end{equation}%
where $\lambda \geq 0$. We will distinguish below the case $\lambda =0$
(gradient)\ and $\lambda >0$ (damped). Both cases belong to a class of
generalized FK models, discussed in Section \ref{s:examples}, with the
standard 1d FK\ model being a special case. In the damped case, we will
further require that $\lambda $ is not large enough, so that the partial
order introduced by Baesens and MacKay \cite{Baesens04}, \cite{Baesens05}
holds.\ Let%
\begin{equation*}
B=\max_{0\leq x,y,z\leq 1}(L_{22}(x,y)+L_{11}(y,z))\text{.}
\end{equation*}

We will assume the following (trivially holding for $\lambda =0$) \textit{%
overdamped} condition:

\begin{description}
\item[(L6)] Assume $4\lambda B\leq 1$.
\end{description}

\subsection{Dynamics of (\protect\ref{r:ordering2}) as a bounded Lattice EDS.%
}

We will be interested in solutions $0\leq u\leq 1$. Because of periodicity
of $L$ and (L3), $w_{o},w_{1}\in X_{\lambda }$ as defined below are stable
equilibria:%
\begin{eqnarray*}
\lambda &=&0:\qquad w_{0}\equiv 0,w_{1}\equiv 0, \\
\lambda &>&0:\qquad w_{0}\equiv (0,0),w_{1}\equiv (1,0),
\end{eqnarray*}%
where in the case $\lambda >0$ we consider as usually at a single lattice
point the evolution of $w(\alpha )=(u(\alpha ),\partial _{t}u(\alpha ))$.
Analogously to Eckmann and Rougemont, we will be interested in solutions of (%
\ref{r:overdamped}) "jumping" at each lattice point infinitely times between 
$w_{0}$ and $w_{1}$. Thus the following choice of the state space is
appropriate: in the case $\lambda =0$, we will assume that for all $\alpha
\in Z^{N}$, 
\begin{equation}
0\leq u(\alpha )\leq 1\text{,}  \label{r:ordering1}
\end{equation}%
and in the case $\lambda >0$, we will in addition require that the
derivative $v(\alpha )=\partial _{t}u(\alpha )$ is not too large:%
\begin{equation}
0\leq 2\lambda \cdot v(\alpha )+u(\alpha )\leq 1.  \label{r:ordering2}
\end{equation}

The state spaces will thus be as follows:

\begin{itemize}
\item $\lambda =0$: $X_{0}=\{u:\mathbb{Z}^{N}\rightarrow \mathbb{R}$
satisfies (\ref{r:ordering1})$\},$

\item $\lambda >0$: $X_{\lambda }=\{(u,v):\mathbb{Z}^{N}\times \mathbb{Z}%
^{N}\rightarrow \mathbb{R}$ satisfies (\ref{r:ordering1}),(\ref{r:ordering2})%
$\}$.
\end{itemize}

Clearly by Tychonoff theorem, for all $\lambda \geq 0$, $X_{\lambda }$ is a
compact set in the induced product topology (i.e. topology of convergence at
each lattice point). By combining results from Section \ref{s:examples} and
the order-preserving property of the solution of (\ref{r:overdamped}), we
see that the theory developed in this paper applies to (\ref{r:overdamped}).

\begin{proposition}
\label{l:boundedEDS}Assume $\lambda \geq 0$ and $L$ is $C^{2}$ satisfying
(L1)-(L6). Then the solution of (\ref{r:overdamped}) generates a bounded
Lattice EDS on $X_{\lambda }$, which is compact in the induced product
topology.
\end{proposition}

\begin{proof}
We first show that $X_{\lambda }$ is invariant. Note that the condition (L4)
implies that the lattice with the dynamics (\ref{r:overdamped}) is
cooperative (see relation (2) in \cite{Baesens05}). In the case $\lambda =0$%
, the ordering on $\mathbb{R}^{\mathbb{Z}^{N}}$ is defined with $u\leq
u^{\ast }$, if for all $\alpha \in \mathbb{Z}^{N}$, 
\begin{equation*}
u(\alpha )\leq u^{\ast }(\alpha )\text{.}
\end{equation*}%
By \cite{Baesens05}, Proposition 2.1, the dynamics (\ref{r:overdamped})
preserves ordering, thus $w_{0}\leq u(0)\leq w_{1}$ implies that for all $%
t\geq 0$, $w_{0}\leq u(t)\leq w_{1}$, so $X_{0}$ is invariant.

Let $\lambda >0$. The partial order on $(\mathbb{R\times R)}^{\mathbb{Z}%
^{N}} $introduced in \cite{Baesens04} is now $(u,v)\leq (u^{\ast },v^{\ast
}) $, if for all $\alpha \in \mathbb{Z}^{N},$%
\begin{eqnarray*}
u(\alpha ) &\leq &u^{\ast }(\alpha ), \\
2\lambda v(\alpha )+u(\alpha ) &\leq &2\lambda v^{\ast }(\alpha )+u^{\ast
}(\alpha )\text{.}
\end{eqnarray*}

We set as usual $v(\alpha )=\partial _{t}u(\alpha )$. By \cite{Baesens05},
Proposition 4.2 and (L6), the equation (\ref{r:overdamped}) preserves the
partial order above. The conditions (\ref{r:ordering1}), (\ref{r:ordering2})
can be read as $w_{0}\leq (u(0),\partial _{t}u(\alpha ))\leq w_{1}$, thus by
the order preserving property, for all $t\geq 0$, $w_{0}\leq (u(t),\partial
_{t}u(t))\leq w_{1}$. We conclude that $X_{\lambda }$ is invariant.

We have shown in Section \ref{s:examples} that (\ref{r:overdamped})
generates a Lattice EDS. By compactness and (\ref{r:ordering1}), the energy%
\begin{equation*}
e_{u}(\alpha )=\sum_{j=1}^{N}L(u(\alpha ),u(\alpha +\varepsilon _{j}))
\end{equation*}%
is bounded, thus we have a bounded Lattice EDS.
\end{proof}

\subsection{The set of equilibria.}

The set $\mathcal{E}$ of equilibria of (\ref{r:overdamped}) corresponds in
the case $\lambda =0$ to the solutions of 
\begin{equation*}
\sum_{j=1}^{N}(L_{2}(u(\alpha -\varepsilon _{j}),u(\alpha )+L_{1}(u(\alpha
),u(\alpha +\varepsilon _{j}))=0,
\end{equation*}%
and in the case $\lambda >0$ to the same set with $v=\delta _{t}u\equiv 0$.
For the moment, we consider the case $N=1$, and assume (L4) is somewhat
stronger, i.e. there is $\delta >0$ so that $L_{12}\geq \delta .$ By the
well-known correspondence between the equilibria of generalized 1d FK\
models and orbits of area-preserving, positive twist maps, there exists a
map on the annulus $f:\mathbb{S}^{1}\times \mathbb{R}$ defined with $%
(u,p)\mapsto (u^{\ast },p^{\ast })$, $p^{\ast }=L_{2}(u,u^{\ast }),$ $%
p=-L_{1}(u,u^{\ast })$ (see e.g. \cite{Katok95}, Section 9.3.b).

We will be interested in functions $L$ which generate twist maps with a
phase portrait as the classical 1d pendulum (see Example \ref{e:twistmap}
below), namely the cases where stable and unstable manifolds of the
equilibria $(0,0),(1,0)$ of the twist map consist of their heteroclinic
connections.

\begin{example}
\label{e:twistmap}We choose a smooth potential $V:\mathbb{R}\rightarrow 
\mathbb{R}$, periodic with period $1$, symmetric ($V(x)=V(-x)$), with unique
maximum per period at $x=0$; for example $V(x)=\sin (2\pi x)$. Then we
consider the continuous flow $\varphi $ of $(u,\delta _{t}u)$ of the second
order equation 
\begin{equation*}
\partial _{tt}u(t)=V^{\prime }(u(t))
\end{equation*}%
on the annulus $\mathbb{S}^{1}\times \mathbb{R}$, and let $f_{V}:\mathbb{S}%
^{1}\times \mathbb{R\rightarrow S}^{1}\times \mathbb{R}$ be its time-one
map. As $f_{V}$ is an area-preserving twist map, one can define its
generating function $L_{V}$ which satisfies (L1)-(L6) (see also \cite%
{Katok95}, Section 9.3.b).
\end{example}

This motivates us to define the set of stable or asymptotically stable
equilibria $\mathcal{S}\subset \mathcal{E\subset }X_{\lambda }$ containing $%
w_{0},w_{1}$, and all $v\in X_{\lambda }\,$for which there exists a vector $%
\varepsilon \in \{\pm \varepsilon _{1},...,\pm \varepsilon _{N}\}$ so that 
\begin{equation}
\lim_{n\rightarrow \infty }T^{n\varepsilon }v\in \{w_{0},w_{1}\}
\label{r:asymptotic}
\end{equation}%
(in the product topology). In the Example \ref{e:twistmap}, the set $%
\mathcal{S}$ is connected, so, as we will see, a solution of (\ref%
{r:overdamped}) can "slide" between the equilibria $w_{0}$ and $w_{1}$.

\subsection{Translationally invariant ergodic measures.}

Rather than discussing asymptotic behavior for a particular initial
condition in $X_{\lambda }$, we will prove existence of coarsening dynamics
for almost every initial condition with respect to a particular Borel
probability measure on $X_{\lambda }$. We now specify properties of such
measures and construct some examples. We will always consider Borel
probability measures on $X_{\lambda }$ with respect to the induced product
topology on $X_{\lambda }$.

The group $\mathbb{Z}^{N}$ naturally acts on $X_{\lambda }$ as the group of
translations $T^{\beta }$ for $\beta \in \mathbb{Z}^{N}$:%
\begin{eqnarray*}
\text{For }\lambda &=&0\text{, }(T^{\beta }u)(\alpha )=u(\alpha +\beta ), \\
\text{For }\lambda &>&0\text{, }(T^{\beta }(u,v))(\alpha )=(u,v)(\alpha
+\beta ).
\end{eqnarray*}

We say that a Borel probability measure $\mu $ on $X_{\lambda }$ is \textit{%
translationally invariant}, if for each $\beta \in \mathbb{Z}^{N}\,$and each
measurable $A\subset X_{\lambda }$, $\mu (T^{\beta }(A))=\mu (A)$ (clearly
it is sufficient to show this only for generators $T^{\varepsilon
_{1}},...,T^{\varepsilon _{n}})$. A set $A\subset X_{\lambda }$ is
translationally invariant, if for each $\beta \in \mathbb{Z}^{N}$, $T^{\beta
}(A)=A$. A translationally invariant Borel probability measure $\mu $ is 
\textit{ergodic}, if all measurable translationally invariant sets have
measure $0$ or $1$.

We can also naturally define the \textit{reflection} operator $R$ on $%
X_{\lambda }$, which "swaps" the stable equilibria $0$ and $1$:%
\begin{eqnarray*}
\text{For }\lambda &=&0,\qquad (Ru)(\alpha )=1-u(\alpha ), \\
\text{For }\lambda &>&0,\qquad (R(u,v))(\alpha )=(1-u(\alpha ),-v(\alpha ))%
\text{.}
\end{eqnarray*}%
Note that (L1) and (L5)\ imply that the equation (\ref{r:overdamped}) is $R$%
-invariant. The first two required properties of Borel probability measures $%
\mu $ on $X_{\lambda }$ are the following:

\begin{description}
\item[(M1)] $\mu $ is translationally invariant and ergodic,

\item[(M2)] $\mu $ is $R$-invariant.
\end{description}

There are many Borel probability measures on $X_{\lambda }$ satisfying
(M1),(M2);\ we construct only one, canonical example $\mu _{\lambda }$, $%
\lambda \geq 0$.

Let $\lambda =0$. Let $\nu $ be the Bernoulli measure on $\{0,1\}$, $\nu
(\{0\})=\nu (\{1\})=1/2,$ and let $\mu _{0}=\nu ^{\mathbb{Z}^{N}}$ (defined
in the standard way as a product of countably many probability measures).
This can be understood as randomly and independently putting each lattice
point in one of the minima $0,1$ with the same probability $1/2$. If $%
\lambda >0$, we do the same, and set the initial velocity to be almost
surely $0$. More precisely, $\mu _{\lambda }=\mu _{0}\times \delta _{\{0\}}$%
, where $\delta _{\{0\}}$ is the atomic measure such that $\delta
_{\{0\}}(\{u_{0}\})=1$, $u_{0}\equiv 0$.

Now $\mu _{\lambda }$ are by definition translationally invariant and $R$%
-invariant. It is easy to show that $\mu _{\lambda }$ is ergodic (e.g. by
applying \cite{Bekka00}, Theorem 1.3).

\subsection{Existence of coarsening.}

In addition to (M1), (M2), we also assume the following

\begin{description}
\item[(M3)] For $\mu $-a.e. $u\in X_{\lambda }$, $\omega (u)\cap (\mathcal{E}%
-\mathcal{S})=\emptyset .$
\end{description}

Here $\omega (u)$ is the $\omega $-limit set with respect to dynamics (\ref%
{r:overdamped}). In dimension $N=1$, this means that $L$ generates a twist
map with a phase portrait as in Example \ref{e:twistmap}, and that $\mu $%
-a.e. $u$ does not lie on a stable manifold (with respect to the dynamics of
(\ref{r:overdamped})) of some of spatially periodic or quasiperiodic
equilibria different from $w_{0},w_{1}$. As discussed also in \cite%
{Eckmann98} in the PDE case, this is a reasonable assumption but a
technically difficult claim. In \cite{Slijepcevic13} we will show the
following:

\begin{claim}
The measures $\mu _{\lambda }\,$constructed above satisfy (M3) in dimensions 
$N=1,2$ for the family of functions $L_{V}$ constructed in Example \ref%
{e:twistmap}.
\end{claim}

Here we bypass this difficulty, and prove the following existence of
coarsening dynamics in the general, abstract case, by assuming (M3):

\begin{theorem}
\label{t:coarsening}Assume $\lambda \geq 0$, $L$ is a $C^{2}$ function
satisfying (L1)-(L6), and $\mu $ is a Borel probability measure on $%
X_{\lambda }$ satisfying (M1)-(M3). Then for $\mu $-a.e. $u\in X_{\lambda }$%
, its $\omega $-limit set with respect to (\ref{r:overdamped}) contains both
stable equilibria $w_{0},w_{1}$.
\end{theorem}

\begin{proof}
Let $A_{0}$ be the set of all $u\in X_{\lambda }\,\ $so that $\omega (u)$
contains either $w_{0}$ or an equilibrium $v$ asymptotic to $w_{0}$ in the
sense of (\ref{r:asymptotic}), and let $A_{1}\subset X_{\lambda }$ be an
analogous set with respect to $w_{1}$. By Proposition \ref{l:boundedEDS},
the dynamics of (\ref{r:overdamped}) generates a bounded Lattice EDS\
semiflow on a compact set $X_{\lambda }$, thus Proposition \ref{l:unstable}
applies. So for each $u\in X_{\lambda }$, its $\omega $-limit set contains
an equilibrium. Now because of (M3),%
\begin{equation*}
\mu (A_{0}\cup A_{1})=1\text{.}
\end{equation*}

However $R$-invariance of (\ref{r:overdamped}) and (M2) easily imply that $%
\mu (A_{0})=\mu (A_{1})$. As $A_{0},A_{1}$ are translationally invariant, by
ergodicity, $\mu (A_{0}),\mu (A_{1})\in \{0,1\}.$ We conclude that $\mu
(A_{0})=\mu (A_{1})=1$.$\ \,$Finally, by the translational invariance of $%
A_{0},A_{1}$, relation (\ref{r:asymptotic}), and closedness of $\omega $%
-limit sets, we deduce that each $u\in A_{0}\,$must contain $w_{0}$ in its $%
\omega $-limit set, respectively each $u\in A_{1}$ must contain $w_{1}$ in
its $\omega $-limit set. Now $A_{0}\cap A_{1}$ is the required subset of $%
X_{\lambda }$ of full measure.
\end{proof}

The interpretation of Theorem \ref{t:coarsening} is that for $\mu $-a.e.
initial condition, the size of "droplets" of lattice points in the same
equilibrium $0$ or $1$ grows with time, but a single lattice point changes
its position infinitely many times. By Corollary \ref{t:dissipation}, the
average time a lattice point spends "moving" between $w_{0}$ and $w_{1}$ is $%
0$. One can also use the bounds on relaxation times in Section \ref%
{s:relaxationtimes} to calculate upper bounds on the maximal "lifetime" of
droplets of certain size.

\subsection{Non-equilibrium point in a $\protect\omega $-limit set}

As was mentioned in the introduction, the dynamics of Lattice EDS\
topologically differs from gradient systems, as we can construct initial
conditions $u$ so that $\omega (u)$ contains a non-equilibrium point. This
is possible even for bounded Lattice EDS on compact sets in dimension $N=1$.
We consider the equation (\ref{r:overdamped}), let $\lambda =0$, dimension $%
N=1$, and assume $L$ is of the type $L_{V}$ constructed in Example \ref%
{e:twistmap}, so $L_{V}$ satisfies (L1)-(L6). As discussed in more detail in 
\cite{Gallay:12}, Example 5.7, in the continuous-space case, we consider an
initial condition%
\begin{equation*}
u(\alpha )=\left\{ 
\begin{array}{ll}
0 & b_{2n}\leq |\alpha |<b_{2n+1} \\ 
1 & b_{2n+1}\leq |\alpha |<b_{2n+2,}%
\end{array}%
\right.
\end{equation*}%
where $0=b_{0}<b_{1}<b_{2}<...$ is an increasing sequence of integers,
satisfying also $b_{n+1}\gg b_{n}$. As explained in \cite{Gallay:12}, the
annihilation of kinks in the coarsening sequence always happens at the
origin $\alpha =0$, so for a sequence $b_{n}$ increasing rapidly enough, $%
\omega $-limit set of $u$ in $X_{0}$ contains both $w_{0},w_{1}$ but also
non-equilibrium, heteroclinic orbits of (\ref{r:overdamped}) connecting
them. A detailed discussion and a proof will appear in \cite{Slijepcevic13}.

\section{Appendix:\ Proofs of Lemmas \protect\ref{l:ricatti1}, \protect\ref%
{l:ricatti2}}

Here we study the recurrent sequence%
\begin{equation}
G_{r+1}=G_{r}-\lambda r^{N-1}+\varepsilon \frac{G_{r}^{2}}{r^{N-1}},
\label{r:sequence}
\end{equation}%
such that $G_{r}/r^{N-1}\,\ $is bounded. We substitute $g_{r}=G_{r}/r^{N-1}$%
, and can study it as orbits of a discrete dynamical system $G(r,g)=(r^{\ast
},g^{\ast })$,%
\begin{eqnarray}
r^{\ast } &=&r+1,  \label{r:orbitr} \\
g^{\ast } &=&(g-\lambda +\varepsilon g^{2})\left( \frac{r}{r+1}\right)
^{N-1},  \label{r:orbitg}
\end{eqnarray}%
where $\lambda ,\varepsilon >0$ are parameters. $G$ is defined on $%
[1,+\infty )\times \mathbb{R\rightarrow }[1,+\infty )\times \mathbb{R}$. We
want to give sufficient and sharp conditions on $(r,g)$ so that the second
variable is bounded along the orbit of $G$. We first note that for $r=\infty 
$, $G$ has two fixed points $\pm \sqrt{\lambda /\varepsilon }$. We will see
that the fixed point $+\sqrt{\lambda /\varepsilon }$ is asymptotically a
saddle point, and that the necessary and sufficient condition for $g\geq 0$
to be bounded\ along the orbit of $G$ is that $(r,g)\,$is below the stable
manifold of $+\sqrt{\lambda /\varepsilon }$. More precisely:

\begin{lemma}
\label{l:unstable}There exist a continuos curve $g_{s}:[1,+\infty )\times 
\mathbb{R\cup \{+\infty \}}$ such that:

\begin{description}
\item[(i)] If $g=g_{s}(r)$, then $\lim_{n\rightarrow \infty
}G^{n}(r,g)=(+\infty ,\sqrt{\lambda /\varepsilon })$.

\item[(ii)] If $0\leq g<g_{s}(r)$, then $\lim_{n\rightarrow \infty
}G^{n}(r,g)=(+\infty ,-\sqrt{\lambda /\varepsilon })$.

\item[(iii)] If $g>g_{s}(r)$, then $\lim_{n\rightarrow \infty
}G^{n}(r,g)=(+\infty ,+\infty )$.

\item[(iv)] The curve $g_{s}(r)$ is non-increasing. If $N\geq 2$, $g_{s}(r)$
is strictly decreasing.
\end{description}
\end{lemma}

Note that we did not exclude the possibility that the unstable manifold $%
g_{s}$ diverges for some $r_{0}\geq 1$ (and is then formally $=+\infty $ on $%
[1,r_{0}]$. We will see later that this is not possible.

\begin{proof}
In order to apply standard results on invariant manifolds of fixed points,
we need to choose the right substitution and compactify (\ref{r:orbitr}), (%
\ref{r:orbitg}) at $r=+\infty $. We set the bijective map $[0,1]\rightarrow
\lbrack 1,+\infty ]$, $s\mapsto r$ with%
\begin{equation*}
r=\frac{1}{1-s}.
\end{equation*}

If $i(s,g)=(r,g)$, $i^{-1}(r,g)=(s,g)$, we now analyze the map $%
H:[0,1]\times \mathbb{R\rightarrow }[0,1]\times \mathbb{R}$ such that $%
G=i^{-1}\circ H\circ i$. Then it is easy to check that the map $H$ is given
with $(s,g)\longmapsto (s^{\ast },g^{\ast })$,%
\begin{eqnarray*}
s^{\ast } &=&\frac{1}{2-s}, \\
g^{\ast } &=&(g-\lambda +\varepsilon g^{2})\left( \frac{1}{2-s}\right)
^{N-1}.
\end{eqnarray*}

Clearly $H$ is $C^{1}$ on $[0,1]\times \mathbb{R}$ and that it can be
extended to a $C^{1}$ map on a neighbourhood of $[0,1]\times \mathbb{R}$. As 
$s^{\ast }>s$ on $[0,1)$, the map $H$ has two fixed points $(s,g)=(1,\pm 
\sqrt{\lambda /\varepsilon })$. We calculate the differentials at fixed
points, and get%
\begin{eqnarray*}
DH(1,\sqrt{\lambda /\varepsilon }) &=&\left( 
\begin{array}{cc}
1 & 0 \\ 
(N-1)\sqrt{\lambda /\varepsilon } & 1+2\sqrt{\lambda \varepsilon }%
\end{array}%
\right) , \\
DH(1,-\sqrt{\lambda /\varepsilon }) &=&\left( 
\begin{array}{cc}
1 & 0 \\ 
-(N-1)\sqrt{\lambda /\varepsilon } & 1-2\sqrt{\lambda \varepsilon }%
\end{array}%
\right) .
\end{eqnarray*}

Clearly the fixed point $(1,-\sqrt{\lambda /\varepsilon })$ is a sink:\ the $%
s=1$ direction is attracting, as the direction $s=1$ corresponds to the
eigenvalue $(0,1)$ and the eigenvector $1-2\sqrt{\lambda \varepsilon }$ and
is thus attracting, and the other direction is attracting as $s\mapsto
s^{\ast }$ is increasing. We consider now the fixed point $(1,\sqrt{\lambda
/\varepsilon })$ with eigenvalues $1$, $1+2\sqrt{\lambda \varepsilon }$. We
apply the central manifold theorem, and construct an unstable and a central
manifold. Clearly the unstable manifold is $s=1$. Assume a local central
manifold is given by $(s,g_{s}(s))$, $g_{s}:[1-\delta ,1]\rightarrow \mathbb{%
R}$ (because of the expression for $s\mapsto s^{\ast }$, we can without loss
of generality assume that the central manifold is locally parametrized by $s$%
). One can check that the central manifold is smooth (as $H$ is smooth in
the neighborhood of $(1,\sqrt{\lambda /\varepsilon })$), locally unique, and
that it can be extended to a global\ smooth central manifold $(s,g_{s}(s))$, 
$g_{s}:[s_{0},1]\rightarrow \mathbb{R}$, $0\leq s_{0}<1$ by extending it
with values of $H^{-1}\,$for which $g>0$. Also if $s_{0}>0$, then $%
\lim_{s\rightarrow s_{0}^{+}}g_{s}(s)=+\infty $. Furthermore, as $s\mapsto
s^{\ast }$ is strictly increasing, for each $(s,g_{s}(s))$, $%
\lim_{n\rightarrow \infty }H^{n}(s,g_{s}(s))=(1,\sqrt{\lambda /\varepsilon }%
) $, so the central manifold is asymptotically stable.

We now show that $g_{s}(s)$ is decreasing. From the direction of the central
eigenvector at the fixed point $(1,\sqrt{\lambda /\varepsilon })$, we deduce
that the function $g_{s}(s)$ is eventually decreasing (strictly decreasing
if $N\geq 2$). However, one can easily check that if for any $(s_{1},g_{1})$%
, $(s_{2},g_{2})=H(s_{1},g_{1})$, $(s_{3},g_{3})=H(s_{2},g_{2})$, $g_{2}\geq
g_{1}>0$ implies $g_{3}\geq g_{2}$.(and $g_{3}>g_{2}$ if $s<1$ and $N\geq 2$%
). From this we see that $g_{s}(s)$ is decreasing, and in the case $N\geq 2$
strictly decreasing.

Consider now the negative image of the unstable manifold, i.e. the maximal
set $(s,g_{s}^{-}(s))\,$satisfying $H(s,g_{s}^{-}(s))=(s^{\ast },g_{s}(s))$
and $g_{s}^{-}(s)<0$. Then we can divide the state space $[0,1]\times 
\mathbb{R}$ of the function $H$ into five sets:\ the sets of $(s,g)$ such
that $A=\{g>g_{s}(s)\}$, $W_{s}=\{g=g_{s}(s)\}$, $B=%
\{g_{s}(s)>g>g_{s}^{-}(s)\}$, $W_{s}^{-}=\{g=g_{s}^{-}(s)\}$ and $%
C=\{g_{s}^{-}(s)>g\}$. One can now summarize the discussion above, and show
the following: $H(A)\subset A$, $H(W_{s})\subset W_{s}$, $H(B)\subset
B,\,H(W_{s}^{-})\subset W_{s},$ $H(C)\subset A$, and $(1,\sqrt{\lambda
/\varepsilon })\in W_{s}$, $(1,-\sqrt{\lambda /\varepsilon })\in B$. As the
dynamics of $H$ asymptotically coincides with the dynamics along the
invariant set $s=1$, we get the following:%
\begin{eqnarray*}
(s,g) &\in &A\cup C\Rightarrow \lim_{n\rightarrow \infty
}H^{n}(s,g)=(1,+\infty ), \\
(s,g) &\in &W_{s}\cup W_{s}^{-}\Rightarrow \lim_{n\rightarrow \infty
}H^{n}(s,g)=(1,\sqrt{\lambda /\varepsilon }), \\
(s,g) &\in &B\Rightarrow \lim_{n\rightarrow \infty }H^{n}(s,g)=(1,-\sqrt{%
\lambda /\varepsilon }).
\end{eqnarray*}

By substituting back $r$ instead of $s$ and writing $g_{s}(r)=g_{s}(s(r))$,
these conclusions imply all claims of the Lemma.
\end{proof}

We now prove Lemma \ref{l:ricatti1}. We now see that it is necessary and
sufficient to give upper bounds on $G_{r}$ on the stable manifold, that
means of the form $G_{r}=g_{s}(r)r^{N-1}$. The reason is that by Lemma \ref%
{l:unstable}, if $G_{r}>g_{s}(r)r^{N-1}$, $\lim_{r\rightarrow \infty
}G_{r}/r^{N-1}=+\infty $. Also if $0<G_{r}\leq g_{s}(r)r^{N-1}$, $%
\lim_{r\rightarrow \infty }G_{r}/r^{N-1}=\pm \sqrt{\lambda /\varepsilon }$.

In the case $N=1$, the stable manifold can be expressed exactly and is $%
g_{s}(r)=\sqrt{\lambda /\varepsilon }$, which implies (\ref{r:bound1}).

We deduce the case $N\geq 2$ from the fact that $g_{s}(r)$ is strictly
decreasing, thus $g_{s}(r+1)<g_{s}(r)$. That means%
\begin{equation*}
g_{s}(r)-\lambda +\varepsilon g_{s}(r)^{2}<g_{s}(r)\left( \frac{r+1}{r}%
\right) ^{N-1}\text{.}
\end{equation*}%
If $C_{r}=\left( \frac{r+1}{r}\right) ^{N-1}-1$, by solving the quadratic
inequality and applying $\sqrt{1+x}<1+\sqrt{x}$ for $x>0$ we get%
\begin{equation}
g_{s}(r)<\frac{C_{r}+C_{r}\sqrt{1+4\varepsilon \lambda /C_{r}^{2}}}{%
2\varepsilon }<\frac{C_{r}}{\varepsilon }+\sqrt{\frac{\lambda }{\varepsilon }%
}\text{.}  \label{r:finalone}
\end{equation}%
However%
\begin{equation}
C_{r}<\frac{(N-1)}{r}\left( 1+\frac{1}{r}\right) ^{N-2}\text{.}
\label{r:finaltwo}
\end{equation}%
Combining (\ref{r:finalone}), (\ref{r:finaltwo}) and $G_{r}\leq
g_{s}(r)r^{N-1}$ we get (\ref{r:boundall}). Also $g_{s}(r)$ is finite for
all $r\in \lbrack 1,+\infty )$. One can also show by constructing
appropriate examples, that (\ref{r:boundall}) is up to a constant the best
upper bound in cases $N\geq 3$.

The case $N=2$ and Lemma \ref{l:ricatti2} is, however, more subtle. Fix $N=2$%
, and assume $G_{r}=g_{s}(r)r$. As $G_{r}\geq r\sqrt{\lambda /\varepsilon }%
>0 $, we divide (\ref{r:sequence}) by $G_{r}^{2}$ and get%
\begin{equation}
\frac{G_{r+1}}{G_{r}^{2}}-\frac{1}{G_{r}}=-\lambda \frac{r}{G_{r}^{2}}+\frac{%
\varepsilon }{r}.  \label{r:inverse}
\end{equation}

Inserting $G_{r}\geq r\sqrt{\lambda /\varepsilon }$ in (\ref{r:sequence}),
we get $G_{r+1}\geq G_{r}$. However, as $g_{s}(r)$ is decreasing, $%
G_{r+1}/G_{r}=g_{s}(r+1)/g_{s}(r)\cdot (1+1/r)\leq (1+1/r)\leq 2$. We now
have%
\begin{equation}
\frac{G_{r+1}}{G_{r}^{2}}-\frac{1}{G_{r}}=\frac{G_{r+1}-G_{r}}{G_{r}^{2}}%
\leq 2\frac{G_{r+1}-G_{r}}{G_{r}G_{r+1}}=2\left( \frac{1}{G_{r}}-\frac{1}{%
G_{r+1}}\right) \text{.}  \label{r:difference}
\end{equation}

We now fix an integer $k\geq 2$, and sum (\ref{r:inverse}) for $s=r,...,kr-1$%
. Applying (\ref{r:difference}) and $G_{r+s}\geq G_{r}$ we have%
\begin{eqnarray*}
\frac{1}{G_{r}} &\geq &\frac{1}{G_{r}}-\frac{1}{G_{kr}}\geq -\frac{\lambda }{%
2G_{r}^{2}}\sum_{s=r}^{kr-1}s+\frac{\varepsilon }{2}\sum_{s=r}^{kr-1}\frac{1%
}{s} \\
&\geq &-\frac{\lambda }{4G_{r}^{2}}k^{2}r^{2}+\frac{\varepsilon }{2}\log k%
\text{.}
\end{eqnarray*}

Thus we have reduced the problem of bounding $1/G_{r}$ from below to solving
a quadratic inequality; and then optimizing it for $k$. In the case%
\begin{equation}
\frac{1}{2}\lambda \varepsilon k^{2}r^{2}\log k\leq 1,  \label{r:condition}
\end{equation}
we get by applying $\sqrt{1+x}>1+x/4$ for $0\leq x\leq 1$: 
\begin{equation}
\frac{1}{G_{r}}\geq \frac{-1+\sqrt{1+\lambda \varepsilon k^{2}r^{2}\log k/2}%
}{\lambda k^{2}r^{2}/2}\geq \frac{1}{4}\varepsilon \log k\text{.}
\label{r:final}
\end{equation}

In the case $r^{2}\lambda \varepsilon \leq 1/2$ we choose a positive integer 
$k$ so that%
\begin{equation*}
\left( \frac{1}{2r^{2}\lambda \varepsilon }\right) ^{1/3}\leq k\leq 2\left( 
\frac{1}{2r^{2}\lambda \varepsilon }\right) ^{1/3}\text{.}
\end{equation*}%
As $k^{2}\log k\leq k^{3}/2$, (\ref{r:condition}) holds. Inserting $\log
k\geq -\log (2\lambda \varepsilon r^{2})/3$ into (\ref{r:final}) completes
the proof of Lemma \ref{l:ricatti2}.

\begin{acknowledgement}
The author wishes to thank Th. Gallay for co-developing the PDE equivalent
of the theory, and C. Baesens for useful discussions. The work has been
partially supported by the grant No 037-0372791-2803 of the Croatian
Ministry of Science.
\end{acknowledgement}

\end{document}